\documentclass{article}
\usepackage[utf8]{inputenc}
\usepackage[T1]{fontenc}
\usepackage{lmodern}
\usepackage{babel}
\usepackage{lipsum}
\usepackage{graphicx}
\usepackage{amsmath}
\usepackage{amsfonts}
\usepackage{subcaption}
\usepackage{tabularx}
\usepackage{siunitx}
\usepackage{booktabs}
\usepackage{tikz}
\usetikzlibrary{calc}
\usepackage[ruled,linesnumbered]{algorithm2e}
\usepackage{amsthm}
\usepackage{accents}
\usepackage{parskip}
\usepackage{amssymb}
\usepackage[hyperfootnotes=false]{hyperref}
\usepackage{thmtools}
\usepackage{multirow}
\usepackage{alphabeta} 
\usepackage{thm-restate}

\newtheorem{theorem}{Theorem}[section]
\newtheorem{lemma}[theorem]{Lemma}
\newtheorem{corollaryinner}{Corollary}[theorem]
\newenvironment{corollary}[1][]
 {
  \if\relax\detokenize{#1}\relax
  \else
    \ifcsname #1-used\endcsname
      \expandafter\xdef\csname #1-used\endcsname{\the\numexpr\csname #1-used\endcsname+1}%
    \else
      \expandafter\gdef\csname #1-used\endcsname{1}%
    \fi
    \renewcommand{\thecorollaryinner}{\ref{#1}.\csname #1-used\endcsname}%
  \fi
  \corollaryinner
 }
 {\endcorollaryinner}
\usepackage{arxiv}

\newcommand{\powerset}{\ensuremath{\mathcal{P}}\xspace}

\newcommand{\integerpoints}{\ensuremath{\mathcal{I}_\mathcal{X}}\xspace}
\newcommand{\goodcutpoints}{\ensuremath{\mathcal{GC}_\mathcal{X}}\xspace}
\newcommand{\intsupcutpoints}[1]{\ensuremath{\mathcal{ISC}^{#1}_\mathcal{X}}\xspace}
\newcommand{\objparalcutpoints}[1]{\ensuremath{\mathcal{OPC}^{#1}_\mathcal{X}}\xspace}

\newcommand{\goodcut}{\ensuremath{\mathcal{GC}}\xspace}
\newcommand{\objparalcut}[1]{\ensuremath{\mathcal{OPC}^{#1}}\xspace}
\newcommand{\intsupcut}[1]{\ensuremath{\mathcal{ISC}^{#1}}\xspace}

\newcommand{\coefficients}{\ensuremath{\boldsymbol{\alpha}}\xspace}
\newcommand{\objective}{\ensuremath{\mathbf{c}}\xspace}

\newcommand{\mip}{\ensuremath{\mathsf{P}(a,d)}\xspace}
\newcommand{\mipobjective}{\ensuremath{\mathbf{c}_{\mip}}\xspace}

\newcommand{\maxa}[1]{\ensuremath{{\normalfont \texttt{a}_{\text{max}}}(#1)}\xspace}

\newcommand{\lambdalb}[2]{\ensuremath{\lambda_{lb}(#1,#2)}\xspace}
\newcommand{\lambdaub}[2]{\ensuremath{\lambda_{ub}(#1,#2)}\xspace}
\newcommand{\region}{\ensuremath{\mathcal{R}_{\goodcut}}\xspace}
\newcommand{\regionfunc}[2]{\ensuremath{\texttt{r}_{\goodcut}(#1,#2)}\xspace}
\newcommand{\regionfuncnoargs}{\ensuremath{\texttt{r}_{\goodcut}}\xspace}

\newcommand{\interval}[2]{\ensuremath{\mathbf{I}(#1,#2)}\xspace}
\newcommand{\maxaprime}[2]{\ensuremath{\widehat{{\normalfont \texttt{a}_{\text{max}}}}(#1,#2)}\xspace}
\newcommand{\epsinterval}{\ensuremath{\hat{\epsilon}}\xspace}

\newcommand{\cutsadded}{\ensuremath{\mathcal{S}}\xspace}
\newcommand{\cuts}{\ensuremath{\mathcal{S}'}\xspace}

\newcommand{\lambdadis}{\ensuremath{\Lambda}\xspace}
\newcommand{\lambdadisfull}{\ensuremath{\{ \lambda_1, \ldots, \lambda_{|\lambdadis|}\}}\xspace}
\newcommand{\lambdadisgaps}{\ensuremath{\tilde{\Lambda}}\xspace}
\newcommand{\lambdadisgapsfull}{\ensuremath{\{ [0, \lambda_1) \cup (\lambda_1, \lambda_2) \cup \cdots \cup (\lambda_{n-1}, \lambda_n) \cup  (\lambda_{|\lambdadis|}, 1]\}}\xspace}

\newcommand{\reversemaxaprime}[2]{\ensuremath{\tilde{a_{max}}(#1,#2)}\xspace}
\newcommand{\reverseeps}{\ensuremath{\tilde{\epsilon}}\xspace}

\newcommand{\reals}{\ensuremath{\mathbb{R}}\xspace}
\newcommand{\positivereals}{\ensuremath{\mathbb{R}_{\geq 0}}\xspace}
\newcommand{\ints}{\ensuremath{\mathbb{Z}}\xspace}
\newcommand{\naturals}{\ensuremath{\mathbb{N}}\xspace}
\newcommand{\expectation}{\ensuremath{\mathbb{E}}\xspace}

\newcommand{\normal}[2]{\ensuremath{\mathcal{N}_{4}(#1,#2)}\xspace}

\newcommand{\lpoptimal}{\ensuremath{\mathbf{x}^{LP}}\xspace}
\newcommand{\lpoptimali}[1]{\ensuremath{x^{LP}_{#1}}\xspace}
\newcommand{\incumbent}{\ensuremath{\hat{\mathbf{x}}}\xspace}

\newcommand{\isp}[1]{\ensuremath{\texttt{isp}(#1)}\xspace}
\newcommand{\obp}[2]{\ensuremath{\texttt{obp}(#1,#2)}\xspace}
\newcommand{\eff}[2]{\ensuremath{\texttt{eff}(#1,#2)}\xspace}
\newcommand{\dcd}[3]{\ensuremath{\texttt{dcd}(#1,#2,#3)}\xspace}
\newcommand{\effs}[3]{\ensuremath{\texttt{eff'}(#1,#2,#3)}\xspace}
\newcommand{\dcds}[4]{\ensuremath{\texttt{dcd'}(#1,#2,#3,#4)}\xspace}

\hypersetup{
    colorlinks=true,
    linkcolor=blue,
    filecolor=magenta,      
    urlcolor=cyan,
    citecolor=blue
    }

\def\hat{\mathaccent "705E\relax}

\title{Adaptive Cut Selection in Mixed-Integer Linear Programming}
\author{Mark Turner, Thorsten Koch, Felipe Serrano, Michael Winkler}
\date{\today}

\begin{document}

\author{ \href{https://orcid.org/0000-0001-7270-1496}{\includegraphics[scale=0.06]{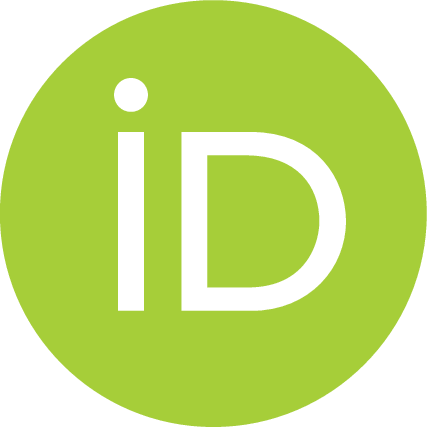}}\hspace{1mm}Mark Turner\thanks{Chair of Software and Algorithms for Discrete Optimization, Institute of Mathematics, Technische Universit{\"a}t Berlin, Straße des 17. Juni 135, 10623 Berlin, Germany}\hspace{2mm}\thanks{Zuse Institute Berlin, Department of Mathematical Optimization, Takustr. 7, 14195 Berlin} \\
	\texttt{turner@zib.de} \\
	\And
	\href{https://orcid.org/0000-0002-1967-0077}{\includegraphics[scale=0.06]{orcid_id_icon.eps}}\hspace{1mm}Thorsten Koch\footnotemark[1]\hspace{2mm}\footnotemark[2]
	\\
	\texttt{koch@zib.de} \\
	\And
	\href{https://orcid.org/0000-0002-7892-3951}{\includegraphics[scale=0.06]{orcid_id_icon.eps}}\hspace{1mm}Felipe Serrano\thanks{I$^2$DAMO GmbH, Englerallee 19, 14195 Berlin, Germany}\hspace{2mm}\footnotemark[2] \\
	\texttt{serrano@zib.de} \\
	\And
	Michael Winkler\thanks{Gurobi GmbH, Ulmenstr. 37-39, 60325 Frankfurt am Main, Germany}\hspace{2mm}\footnotemark[2] \\
	\texttt{winkler@gurobi.com}
}



\maketitle

\begin{abstract}
    Cutting plane selection is a subroutine used in all modern mixed-integer linear programming solvers with the goal of selecting a subset of generated cuts that induce optimal solver performance. These solvers have millions of parameter combinations, and so are excellent candidates for parameter tuning. Cut selection scoring rules are usually weighted sums of different measurements, where the weights are parameters. We present a parametric family of mixed-integer linear programs together with infinitely many family-wide valid cuts. Some of these cuts can induce integer optimal solutions directly after being applied, while others fail to do so even if an infinite amount are applied. We show for a specific cut selection rule, that any finite grid search of the parameter space will always miss all parameter values, which select integer optimal inducing cuts in an infinite amount of our problems. We propose a variation on the design of existing graph convolutional neural networks, adapting them to learn cut selection rule parameters. We present a reinforcement learning framework for selecting cuts, and train our design using said framework over MIPLIB 2017 and a neural network verification data set. Our framework and design show that adaptive cut selection does substantially improve performance over a diverse set of instances, but that finding a single function describing such a rule is difficult. Code for reproducing all experiments is available at \url{https://github.com/Opt-Mucca/Adaptive-Cutsel-MILP}.
\end{abstract}

\section{Introduction}

A Mixed-Integer Linear Program (MILP) is an optimisation problem that is classically defined as:
\begin{align}
    \underset{\mathbf{x}}{\text{argmin}}\{\mathbf{c}^{\intercal}\mathbf{x} \;\; | \;\; \mathbf{A}\mathbf{x} \leq \mathbf{b}, \;\; \mathbf{l} \leq \mathbf{x} \leq \mathbf{u}, \;\; \mathbf{x} \in \mathbb{Z}^{|\mathcal{J}|} \times \mathbb{R}^{n - |\mathcal{J}|} \} \label{eq:mip}
\end{align}
Here, $\mathbf{c} \in \reals^{n}$ is the objective coefficient vector, $\mathbf{A} \in \reals^{m \times n}$ is the constraint matrix, $\mathbf{b} \in \reals^{m}$ is the right hand side constraint vector, $\mathbf{l}, \mathbf{u} \in \reals \cup \{-\infty, \infty\}^{n}$ are the lower and upper variable bound vectors, and $\mathcal{J} \subseteq \{1 , \dots , n\}$ is the set of indices of integer variables. 

One of the main techniques for solving MILPs is the branch-and-cut algorithm, see  \cite{achterberg2007constraint} for an introduction. 
Generating cutting planes, abbreviated as \textit{cuts}, is a major part of this algorithm, and is one of the most powerful techniques for quickly solving MILPs to optimality, see \cite{achterberg2013mixed}. A cut is an inequality that does not remove any feasible solutions of \eqref{eq:mip} when added to the formulation. We restrict ourselves to linear cuts in this paper, and denote a cut as $\coefficients = (\alpha_0, \cdots, \alpha_n) \in \reals^{n+1}$, and denote the set of feasible solutions as \integerpoints, to formally define a cut in \eqref{eq:cut}.
\begin{align}
    \sum_{i=1}^{n}\alpha_i x_i \leq \alpha_0, \;\; \forall x \in \integerpoints, \;\; \text{where} \;\; \mathbf{x}=(x_{1}, \cdots, x_{n}) \label{eq:cut}
\end{align}
The purpose of cuts is to tighten the linear programming (LP) relaxation of \eqref{eq:mip}, where the LP relaxation is obtained by removing all integrality requirements. Commonly, cuts are found that separate the current feasible solution to the LP relaxation, referred to as \lpoptimal, from the tightened relaxation, and for this reason algorithms that find cuts are often called \textit{separators}. This property is defined as follows:
\begin{align}
    \sum_{i=1}^{n}\alpha_i \lpoptimali{i} > \alpha_0, \;\; \text{where} \;\; \lpoptimal = (\lpoptimali{1}, \dots, \lpoptimali{n}) \label{eq:separator}
\end{align}
Within modern MILP solvers, the cut aspect of the branch-and-cut algorithm is divided into cut generation and cut selection subproblems. The goal of cut generation is finding cuts that both tighten the LP relaxation at the current node and improve overall solver performance. The cut selection subproblem is then concerned with deciding which of the generated cuts to add to the formulation \eqref{eq:mip}. That is, given the set of generated cuts $\cuts = \{\coefficients_1, \cdots, \coefficients_{|\cuts|}\}$, find a subset $\cutsadded \subseteq \cuts$ to add to the formulation \eqref{eq:mip}.

We focus on the cut selection subproblem in this paper, where we motivate the need for instance-dependent cut selection rules as opposed to fixed rules, and introduce a reinforcement learning (RL) framework for learning parameters of such a rule. The cut selection subproblem is important, as adding either all or none of the generated cuts to the LP usually results in poor solver performance. This is due to the large computational burden of solving larger LPs at each node when all cuts are added, and the large increase in nodes needed to solve MILPs when no cuts are added. For a summary on MILPs we refer readers to \cite{achterberg2007constraint}, for cutting planes \cite{marchand2002cutting}, for cut selection \cite{wesselmann2012implementing}, and for reinforcement learning \cite{sutton2018reinforcement}.

The rest of the paper is organised as follows. In Section \ref{sec:related_work}, we summarise existing literature on learning cut selection. In Section \ref{sec:theorem}, with an expanded proof in Appendix \ref{sec:proof}, we motivate the need for adaptive cut selection by showing worst case performance of fixed cut selection rules. This section was inspired by \cite{balcan2018learning}, which proved complexity results for fixed branching rules. In Section \ref{sec:scip_cut} we summarise how cut selection is performed in the MILP solver SCIP \cite{scip8}. In Section \ref{sec:rl} we show how to formulate cut selection as a Markov decision process, and phrase cut selection as a reinforcement learning problem. This section was motivated by \cite{gasse2019exact}, which presented variable selection as a Markov decision process as well as experimental results of an imitation learning approach. Finally, in Section \ref{sec:experiments}, we present a thorough computational experiment on learning cut selector parameters that improve root node performance, and study the generalisation of these parameters to the larger solving process. All experiments are done over MIPLIB 2017 \cite{miplib} and a neural network verification data set \cite{nair2020solving} using the MILP solver SCIP version 8.0.1 \cite{scip8}.

\section{Related Work} \label{sec:related_work}

Several authors have proposed cut selection rules and performed several computational studies. The thesis \cite{achterberg2007constraint} presents a linear weighted sum cut selection rule, which drastically reduces solution time to optimality by selecting a reduced number of good cuts. This cut selection rule and algorithm, see \cite{scip8}, can still be considered the basis of what we use in this paper. A more in-depth guide to cutting plane management is given in \cite{wesselmann2012implementing}. Here, a large variety of cut measures are summarised and additional computational results given that show how a reduced subset of good cuts can drastically improve solution time. A further computational study, focusing on cut selection strategies for zero-half cuts, is presented in \cite{zerohalf}. They hypothesise that generating a large amount of cuts followed by heuristic selection strategy is more effective than generating a few deep cuts. Note that the solver and cut selection algorithms used in \cite{achterberg2007constraint}, \cite{wesselmann2012implementing}, and \cite{zerohalf} are different. More recently, \cite{dey2018theoretical} summarises the current state of separators and cut selection in the literature, and poses questions aimed to better develop the science of cut selection. The final remark of the paper ponders whether machine learning can be used to answer some of the posed questions. 

Recently, the intersection of mixed-integer programming and machine learning has received a lot of attention, specifically when it comes to branching, see \cite{balcan2018learning, gasse2019exact, nair2020solving} for examples. To the best of our knowledge, however, there are currently only four publications on the intersection of cut selection and machine learning. Firstly, \cite{balcan2021sample} shows how cut selection parameter spaces can be partitioned into regions, such that the highest ranking cut is invariant to parameter changes within the regions. These results are extended to the class of Chv\`{a}tal-Gomory cuts applied at the root, with a sample complexity guarantee of learning cut selection parameters w.r.t. the resultant branch and bound tree size. Secondly, \cite{tang2020reinforcement} presents a reinforcement learning approach using evolutionary strategies for ranking Gomory cuts via neural networks. They show that their method outperforms standard measures, e.g. max violation, and generalises to larger problem sizes within the same class. Thirdly, \cite{baltean2019scoring} train a neural network to rank linear cuts by expected objective value improvement when applied to a semi-definite relaxation. Their experiments show that substantial computational time can be saved when using their approximation, and that the gap after each cut selection round is very similar to that found when using the true objective value improvement. Most recently, \cite{huang2021learning} proposes a multiple instance learning approach for cut selection. They learn a scoring function parameterised as a neural network, which takes as input an aggregated feature vector over a bag of cuts. Their features are mostly composed of measures normally used to score cuts, e.g. norm violation. Cross entropy loss is used to train their network by labelling the bags of cuts before training starts. 

Our contribution to the literature is three-fold. First, we provide motivation for instance-dependent cut selection by proving the existence of a family of parametric MILPs together with an infinite amount of family-wide valid cuts. Some of these cuts can induce integer optimal solutions directly after being applied, while others fail to do so even if an infinite amount are applied. Using a basic cut selection strategy and a pure cutting plane approach, we show that any finite grid search of the cut selector's parameter space, will miss all parameter values, which select integer optimal inducing cuts in an infinite amount of our instances. An interactive version of this constructive proof is provided in Mathematica\textsuperscript{\tiny\textregistered} \cite{Mathematica}, and instance creation algorithms are provided using SCIP's Python API \cite{scip8,pyscipopt}. Second, we introduce a RL framework for learning instance-dependent cut selection rules, and present results on learning parameters to SCIP's default cut selection rule \cite{achterberg2007constraint} over MIPLIB 2017 \cite{miplib} and a neural network verification data set \cite{nair2020solving}. Third and finally, we implemented a new cut selector plugin, which is available from SCIP 8.0 \cite{scip8}, and enables users to include their own cut selection algorithms in the larger MILP solving process.

\section{Motivating Adaptive Cut Selection} \label{sec:theorem}

This section introduces a simplified cut scoring rule, and discusses how the parameters for such a rule are traditionally set in solvers. A theorem is then introduced that motivates the need for adaptive cut scoring rules, and is proven in Appendix \ref{sec:proof} using a simulated pure cutting plane approach.


Consider the following simplified version of SCIP's default cut scoring rule (see Section \ref{sec:scip_cut} for the default scoring rule):
\begin{align}
    \texttt{simple\_cut\_score}(\lambda, \coefficients, \objective) := \lambda * \isp{\coefficients} + (1- \lambda) * \obp{\coefficients}{\objective}, \quad \lambda \in [0,1], \coefficients \in \reals^{n+1}, \objective \in \reals^{n}
    \label{eq:simple_cut_rule}
\end{align}
Using the general MILP definition given in \eqref{eq:mip}, we define the cut measures integer support (\texttt{isp}) and objective parallelism (\texttt{obp}) as follows:
\begin{align}
    \isp{\coefficients}& := \frac{\sum_{i \in \mathcal{J}} \texttt{nonzero}(\alpha_i)}{\sum_{i=1}^{n} \texttt{nonzero}(\alpha_i)}  \text{, where} \quad \texttt{nonzero}(\alpha_i) = \left \{
    \begin{aligned}
    &0 && \text{if}\ \alpha_i = 0 \\ 
    &1 && \text{otherwise} 
    \end{aligned} \right. \label{eq:intsup_rule} \\
    \obp{\coefficients}{\objective}& := \vert \frac{\sum_{i=1}^{n} \alpha_i c_i}{\sqrt{\sum_{i=1}^{n} \alpha_{i}^{2}}\sqrt{\sum_{i=1}^{n} c_{i}^{2}}} \vert \label{eq:objparal_rule}
\end{align}

We now introduce Theorem \ref{thm:main}, which refers to the $\lambda$ parameter in \eqref{eq:simple_cut_rule}.

\begin{restatable}{theorem}{maintheorem}
Given a finite discretisation of $\lambda$, an infinite family of MILP instances together with an infinite amount of family-wide valid cuts can be constructed. Using a pure cutting plane approach and applying a single cut per selection round, the infinite family of instances do not solve to optimality for any value in the discretisation, but do solve to optimality for an infinite amount of alternative $\lambda$ values.
\label{thm:main}
\end{restatable}

The general purpose of Theorem \ref{thm:main} is to motivate the need for instance-dependent parameters in the cut selection subroutine. The typical approach for finding the best choice of cut selector parameters, see previous SCIP computational studies \cite{achterberg2007constraint, scip7, scip8}, is to perform a parameter sweep, most often a grid search. A grid search, however, leaves regions unexplored in the parameter space. In our simplified cut scoring rule \eqref{eq:simple_cut_rule}, we have a single parameter, namely $\lambda$, and these unexplored regions are simply intervals. We define \lambdadis, the set of values in the finite grid search of $\lambda$, as follows:
\begin{align*}
    \lambdadis := \lambdadisfull, \;\; \text{where} \;\; 0 \leq \lambda_i < \lambda_{i+1} \leq 1 \quad \forall i \in \{1, \cdots, n-1\}, \quad |\lambdadis| \in \naturals
\end{align*}
The set of unexplored intervals in the parameter space, denoted \lambdadisgaps, is then defined as:
\begin{align*}
    \lambdadisgaps: = \lambdadisgapsfull
\end{align*}

Our goal is to show that for any \lambdadis we can construct an infinite family of MILP instances from Theorem \ref{thm:main}. Together with our infinite amount of family-wide valid cuts and specific cut selection rule, we will show that the solving process does not finitely terminate for any choice of $\lambda$ outside of an interval $(\lambda_{lb}, \lambda_{ub}) \subset \lambdadisgaps$. In effect, this shows that using the same fixed $\lambda$ value over all problems in a MILP solver could result in incredibly poor performance for many problems. This is somewhat expected, as a fixed parameter cannot be expected to perform well on all possible instances, and moreover, cut selection is only a small subroutine in the much larger MILP solving process. Additionally, the instance space of MILPs is non-uniform, and good performance over certain problems may be highly desirable as they occur more frequently in practice. Nevertheless, Theorem \ref{thm:main} provides important motivation for adaptive cut selection. See Appendix \ref{sec:proof} for a complete proof.

\section{Cut Selection in SCIP}
\label{sec:scip_cut}
Until now we have motivated adaptive cut selection in a theoretical manner, by simulating poor performance of fixed cut selector rules in a pure cutting approach. Using this motivation, we now present results of how parameters of a cut selection scoring rule can be learnt, and made to adapt with the input instance. We begin with an introduction to cut selection in SCIP \cite{scip8}.

The official SCIP cut scoring rule \eqref{eq:cut_rule} that has been used since SCIP 6.0 is defined as:
\begin{align}
    \begin{split}
    \texttt{cut\_score}(\boldsymbol{\lambda}, \coefficients, \objective, \lpoptimal, \incumbent):= \lambda_1 * \dcd{\coefficients}{\lpoptimal}{\incumbent} + \lambda_2 * \eff{\coefficients}{\lpoptimal} + \lambda_3 * \isp{\coefficients} + \lambda_4 * \obp{\coefficients}{\objective}
    \label{eq:cut_rule} \\
    \lambda_1 + \lambda_2 + \lambda_3 + \lambda_4 = 1, \quad \lambda_i \geq 0 \quad \forall i \in \{1,2,3,4\}, \quad \boldsymbol{\lambda} = [\lambda_1, \lambda_2, \lambda_3, \lambda_4]
    \end{split}
\end{align}
The measures integer support (\texttt{isp}) and objective parallelism (\texttt{obp}) are defined in \eqref{eq:intsup_rule} and \eqref{eq:objparal_rule}. Using the general MILP definition \eqref{eq:mip}, letting \lpoptimal be the LP optimal solution of the current relaxation, and \incumbent be the current best incumbent solution, we define the cut measures directed cutoff distance (\texttt{dcd}) and efficacy (\texttt{eff}) as follows:
\begin{align}
    \dcd{\coefficients}{\lpoptimal}{\incumbent} &:= \frac{\sum_{i=1}^{n} \alpha_i \lpoptimali{i} - \alpha_{0}}{|\sum_{i=1}^{n} \alpha_i y_i|}, \; \text{where} \quad y = \frac{\incumbent - \lpoptimal}{||\incumbent - \lpoptimal||}\label{eq:dcd_rule} \\
    \eff{\coefficients}{\lpoptimal} &:= \frac{\sum_{i=1}^{n} \alpha_i \lpoptimali{i} - \alpha_0}{\sqrt{\alpha_{1}^{2} + ... + \alpha_{n}^{2}}} \label{eq:eff_rule}
\end{align}

We note that in SCIP the cut selector does not control how many times it itself is called, which candidate cuts are provided, nor the maximum amount of cuts we can apply each round. We reiterate that each call to the selection subroutine is called an \textit{iteration} or \textit{round}. Algorithm \ref{alg:cutsel} gives an outline of the SCIP cut selection rule.

\begin{algorithm}
\caption{SCIP Default Cut Selector (Summarised)} \label{alg:cutsel}
\DontPrintSemicolon
\SetKwInOut{Input}{Input}
\SetKwInOut{Output}{Return}
\Input{\textit{cuts} $\in \reals^{s_{1} \times n}$, \textit{forced\_cuts} $\in \reals^{s_{2} \times n}$, \textit{max\_cuts} $\in \ints_{\geq 0}$, ($s_{1},s_{2}) \in \ints_{\geq 0}^{2}$}
\Output{Sorted array of selected cuts, the amount of cuts selected}
\textit{n\_cuts} $\gets$ $s_{1}$ \tcp{Size of \textit{cuts} array} 
\For{\textit{forced\_cut} in \textit{forced\_cuts}}{
    \textit{cuts, n\_cuts} $\gets$ remove cuts from \textit{cuts} too parallel to \textit{forced\_cut}
}
\textit{n\_selected\_cuts} $\gets$ 0 \\
\textit{selected\_cuts} $\gets$ $\emptyset$ \\
\While{\textit{n\_cuts} $>$ 0 and \textit{max\_cuts} $>$ \textit{n\_selected\_cuts}}{
    \tcp{Scoring done with \eqref{eq:cut_rule}. If no primal, efficacy replaces cutoff distance} 
    \textit{best\_cut} $\gets$ select highest scoring cut remaining in \textit{cuts} \\
    \textit{selected\_cuts} $\gets$ \textit{selected\_cuts} $\cup$ \textit{ best\_cut} \\
    \textit{n\_selected\_cuts} $\gets$ \textit{n\_selected\_cuts} + 1 \\
    \textit{cuts, n\_cuts} $\gets$ remove cuts from \textit{cuts} too parallel to \textit{best\_cut} \\
}
\Return{\textit{forced\_cuts} $\cup$ \textit{selected\_cuts}, $s_2$ + \textit{n\_selected\_cuts}}
\end{algorithm}

The SCIP cut selector rule in Algorithm \ref{alg:cutsel} still follows the major principles presented in \cite{achterberg2007constraint}. Cuts are greedily added by the largest score according to the scoring rule \eqref{eq:cut_rule}. After a cut is added, all other candidate cuts that are deemed too parallel to the added cut are filtered out and can no longer be added to the formulation this round. Forced cuts, which are always added to the formulation, prefilter all candidate cuts for parallelism, and are most commonly one-dimensional cuts or user defined cuts. We note that Algorithm \ref{alg:cutsel} is a summarised version of the true algorithm, and has abstracted some procedures. Certain parameters have also been removed for the sake simplicity, such as those which determine when two cuts are too parallel. We further note that $\boldsymbol{\lambda} = \{0.0, 1.0, 0.1, 0.1\}$ as of SCIP 8.0.

Motivated by work from this paper, users can now define their own cut selection algorithms and include them in SCIP for all versions since SCIP 8.0 \cite{scip8}. Users can do this with a single function interface, bypassing the previous need to modify SCIP source code. For example, users can introduce a cut selection rule with an entirely new scoring rule that replaces \eqref{eq:cut_rule}, or introduce a new filtering mechanism that is not based exclusively on parallelism. We hope that this leads to additional research about cut selection algorithms in modern MILP solvers.

\section{Problem Representation and Solution Architecture}
\label{sec:rl}
We now present our approach for learning cut selector parameters for MILPs. In Subsection \ref{subsec:graph} we describe our encoding of a general MILP instance into a bipartite graph. Subsection \ref{subsec:rl} introduces a framework for posing cut selection parameter choices as a RL problem, with Subsection \ref{subsec:architecture} describing the graph convolutional neural network architecture used as our policy network. Subsection \ref{subsec:training} outlines the training method to update our policy network.

\subsection{Problem representation as a graph} \label{subsec:graph}

The current standard for deep learning representation of a general MILP instance is the constraint-variable bipartite graph as described in \cite{gasse2019exact}. Some extensions to this design have been proposed, see \cite{ding2020accelerating}, as well as alternative non graph embeddings, see \cite{miplib} and \cite{steever2020image}. We use the embedding as introduced in \cite{gasse2019exact} and the accompanying graph convolutional neural network (GCNN) design, albeit with the removal of all LP solution specific features and a different interpretation of the output. The construction process for the bipartite graph can be seen in Figure \ref{fig:bipartite_graph}. 

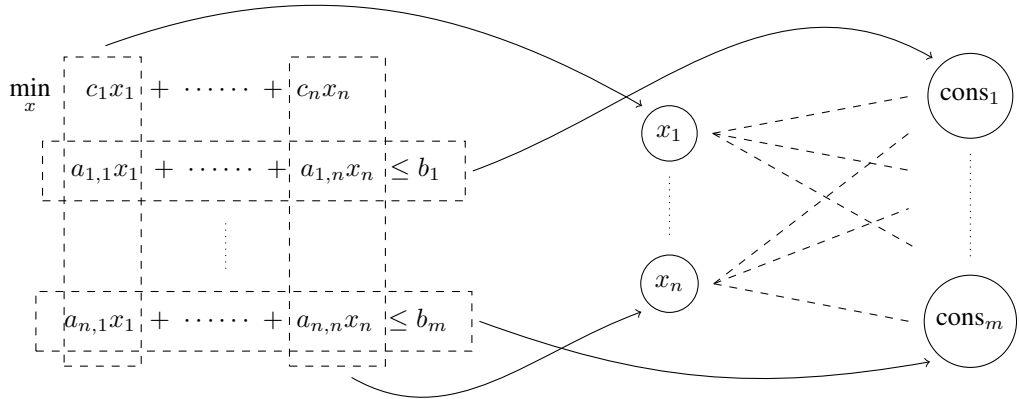
\begin{figure}
\centering
\begin{tikzpicture}
    \draw (2.5,-2) node (A) {$\underset{x}{\text{min}}$ \;\;\; $c_{1}x_{1} \;  + \; \cdots \cdots \; + \; c_{n} x_{n}$};
    \draw (3.5,-3) node (B) {$a_{1,1}x_{1} \; + \; \cdots \cdots \; + \; a_{1,n}x_{n} \; \leq b_{1}$};
    \draw (3.5,-5) node (C) {$a_{n,1}x_{1} \; + \; \cdots \cdots \; + \; a_{n,n}x_{n} \; \leq b_{m}$};
    
    \draw[dotted] ($(B.south)+(-0.4,-0.4)$) -- ($(C.north)+(-0.4,0.4)$);
    
    \draw[dashed] ($(B.north west)+(-0.25,0.1)$) rectangle ($(B.south east)+(0.2,-0.1)$);
    \draw[dashed] ($(C.north west)+(-0.25,0.1)$) rectangle ($(C.south east)+(0.2,-0.1)$);
    
    \draw[dashed] ($(A.north west)+(0.85,0.15)$) rectangle ($(C.south west)+(1.15,-0.3)$);
    \draw[dashed] ($(A.north east)+(-0.95,0.15)$) rectangle ($(C.south east)+(-0.95,-0.3)$);
    
    \node[draw,circle] (x1) at (9,-2.5) {$x_{1}$};
    \node[draw,circle] (xn) at (9,-4.5) {$x_{n}$};
    
    \node[draw,circle] (c1) at (13,-2) {$\text{cons}_{1}$};
    \node[draw,circle] (cm) at (13,-5) {$\text{cons}_{m}$};
    
    \draw[dotted] ($(x1.south)+(0, -0.2)$) -- ($(xn.north) + (0, 0.2)$);
    \draw[dotted] ($(c1.south)+(0, -0.2)$) -- ($(cm.north) + (0, 0.2)$);
    
    \draw[dashed] ($(x1.east)+(0.2,0)$) -- ($(c1.west)+(-0.2,0)$);
    \draw[dashed] ($(x1.east)+(0.2,0)$) -- ($(c1.west)+(-0.2,-1)$);
    \draw[dashed] ($(x1.east)+(0.2,0)$) -- ($(c1.west)+(-0.2,-2)$);
    
    \draw[dashed] ($(xn.east)+(0.2,0)$) -- ($(cm.west)+(-0.2,0)$);
    \draw[dashed] ($(xn.east)+(0.2,0)$) -- ($(cm.west)+(-0.2,1.5)$);
    \draw[dashed] ($(xn.east)+(0.2,0)$) -- ($(cm.west)+(-0.2,2.5)$);
    
    \draw[->] ($(A.north west)+(1.4,0.3)$) to [out=20,in=150] ($(x1.north west)+(-0.1,0.1)$);
    \draw[->] ($(C.south east)+(-1.4,-0.45)$) to [out=330,in=200] ($(xn.south west)+(-0.1,-0.1)$);
    
    \draw[->] ($(B.east)+(0.3,0)$) to [out=20,in=150] ($(c1.north west)+(-0.1,0.1)$);
    \draw[->] ($(C.east)+(0.3,0)$) to [out=340,in=190] ($(cm.south west)+(-0.1,-0.1)$);
\end{tikzpicture}
\caption{A visualisation of the variable-constraint bipartite graph construction from a MILP.}
\label{fig:bipartite_graph}
\end{figure}

The bipartite graph representation can be written as $G = \{\mathbf{V},\mathbf{C},\mathbf{E}\} \in \mathcal{G}$, where $\mathcal{G}$ is the set of all bipartite graph representations of MILP instances. $\mathbf{V} \in \reals^{n \times 7}$ is the feature matrix of nodes on one side of the graph, which correspond one-to-one with the variables (columns) in the MILP. $\mathbf{C} \in \reals^{m \times 7}$ is the feature matrix of nodes on the other side, and correspond one-to-one with the constraints (rows) in the MILP. An edge $(i,j) \in \mathbf{E}$ exists when the variable represented by $x_i$ has non-zero coefficient in constraint $\text{cons}_j$, where $i \in \{1, \dots , n\}$ and $j \in \{1, \dots , m\}$. We abuse notation slightly and say that $\mathbf{E} \in \reals^{m \times n \times 1}$, where $\mathbf{E}$ is the edge feature tensor. Note that we do not extend our MILP representation after every round of cuts is added due to using single-step learning, see Subsection \ref{subsec:rl}. The representation is extendable however to multi-step learning, where the added cuts could become constraints. The exact set of features can be seen in Table \ref{tab:features}.
\begin{table}
\centering
\begin{tabular}{ p{1.3cm} p{8.5cm} p{3cm} }
\hline
Tensor & Features & Value Range \\
\hline
 & Normalised objective coefficient  & [-1, 1] \\
$\mathbf{V}$ & Normalised lower bound $|$ upper bound & $\{-2, [-1,1], 2\}^{2}$ \\
 & Type: binary $|$ integer $|$ continuous $|$ implicit integer & one-hot encoding \\
\hline
 \multirow{3}{*}{$\mathbf{C}$}& Absolute objective parallelism (cosine similarity) & [0,1] \\
 & Normalised RHS per constraint & [-1, 1] \\
 & Type: linear $|$ logicor $|$  knapsack $|$ setppc $|$ varbound & one-hot encoding \\
\hline
$\mathbf{E}$ & Normalised coefficients per constraint & [-1, 1] \\
\hline
\end{tabular}
\caption{Feature descriptions of variable (column) feature matrix $\mathbf{V}$, constraint (row) feature matrix $\mathbf{C}$, and edge feature tensor $\mathbf{E}$.}
\label{tab:features}
\end{table}

\subsection{Reinforcement Learning Framework} \label{subsec:rl}
We formulate our problem as a single step Markov decision process. The initial state of our environment is $s_{0}=G^{0}=G$. An agent takes an action $a_{0} \in \reals^{4}$, resulting in an instant reward $\texttt{r}(s_{0},a_{0}) \in \reals$, and deterministically transitions to a terminal state $s_{1} = G^{N_{r}}$, $N_{r} \in \ints$. The action taken, $a_{0}$, is dictated by a policy $\pi_{\theta}(a_{0}|s_{0})$ that maps any initial state to a distribution over our action space, i.e. $a_{0} \sim \pi_{\theta}(\cdot|s_{0})$. 

The MILP solver in this framework is our environment, and the cut selector our agent. Let $N_{r}$ be the number of paired separation and cut selection rounds we wish to apply, and $G^{i} \in \mathcal{G}$ be the bipartite graph representation of $G \in \mathcal{G}$ after $i$ rounds have been applied. The action $a_{0} \in \reals^{4}$ is the choice of cut selector parameters $\{\lambda_1 , \lambda_2 , \lambda_3 , \lambda_4\}$ followed by $N_{r}$ paired separation rounds. Applying action $a_{0}$ to state $s_{0}$ results in a deterministic transition to $s_{1} = G^{N_{r}}$, defined by the function $\texttt{f} : \mathcal{G} \times \reals^{4} \xrightarrow{} \mathcal{G}$. 

The baseline function, $\texttt{b}(s_{0}) : \mathcal{G} \xrightarrow{} \reals,$ maps an initial state $s_{0}$ to the primal-dual difference of the LP solution of $\texttt{f}(s_{0},a') \in \mathcal{G}$, where the solver is run with standard cut selector parameters, $a' \in \reals^{4}$, and some pre-loaded primal solution. The primal-dual difference in this experiment can be thought of as a strict dual bound improvement, as the pre-loaded primal cannot be improved upon without a provable optimal solution itself. The pre-loaded primal also serves to make directed cutoff distance active from the beginning of the solving process. We do note that this is different to the normal solve process and introduces some bias, most notably for directed cutoff distance. Let $\texttt{g}_{a_{0}}(s_{0})$ be the primal-dual difference of the LP solution of $\texttt{f}(s_{0},a_{0})$ if $a_{0}$ are the cut selector parameter values used. The reward $\texttt{r}(s_{0},a_{0})$ can then be defined as:
\begin{align*}
    \texttt{r}(s_{0},a_{0}) := \frac{\texttt{b}(s_{0})-\texttt{g}_{a_{0}}(s_{0})}{|\texttt{b}(s_{0})| + 10^{-8}}
\end{align*}

Let $(s_{0}, a_{0}, s_{1}) \in \mathcal{G} \times \reals^{4} \times \mathcal{G}$ be a trajectory, also called a roll out in the literature. The goal of reinforcement learning is to maximise the expected reward over all trajectories. That is, we want to find $\theta$ that parameterises:
\begin{align}
    \underset{\theta}{\text{argmax}} \underset{(s_{0}, a_{0}, s_{1}) \sim \pi_{\theta}}{\text{\expectation}} [\texttt{r}(s_{0}, a_{0},s_{1})] \; = \;\underset{\theta}{\text{argmax}} \int_{s_{1} \in \mathcal{G}}\int_{(s_{0},a_{0}) \in \texttt{f}^{-1}(s_{1})} p(s_{0})\pi_{\theta}(a_{0}|s_{0})\texttt{r}(s_{0},a_{0}) \; ds_{0} da_{0} ds_{1} 
    \label{eq:reinforce}
\end{align}
Here, $p(s_{0})$ is the density function on instances $s \in \mathcal{G}$ evaluated at $s=s_{0}$. The pre-image $\texttt{f}^{-1}(s_{1}): \mathcal{G} \xrightarrow{} \reals^{4} \times \mathcal{G}$ is defined as:
\begin{align*}
    \texttt{f}^{-1}(\mathcal{G}') := \{ (s_{0},a_{0}) \in \mathcal{G} \times \reals^{4} \; | \; \texttt{f}(s_{0},a_{0}) \in \mathcal{G}' \}, \quad \mathcal{G}' \subseteq \mathcal{G}
\end{align*}
We note that equation \eqref{eq:reinforce} varies from the standard definition as seen in \cite{sutton2018reinforcement}, and those presented in similar research \cite{gasse2019exact, tang2020reinforcement}, as our action space is continuous. Additionally, as the set $\mathcal{G}$ is infinite and we do not know the density function $p(s)$, we use sample average approximation, creating a uniform distribution around our input data set.

\subsection{Policy Architecture} \label{subsec:architecture}

Our policy network, $\pi_{\theta}(\cdot|s_{0} \in \mathcal{G})$, is parameterised as a graph convolutional neural network, and follows the general design as in \cite{gasse2019exact}, where $\theta$ fully describes the complete set of weights and biases in the GCNN. The changes in design are that we use 32 dimensional convolutions instead of 64 due to our lower dimensional input, and output a 4 dimensional vector as we are interested in cut selector parameters. This technique of using the constraint-variable graph as an embedding for graph neural networks has gained recent popularity, see \cite{cappart2021combinatorial} for an overview of applications in combinatorial optimisation. 

Our policy network takes as input the constraint-variable bipartite graph representation $s_{0} = \{\mathbf{V}, \mathbf{C}, \mathbf{E} \}$. Two staggered half-convolutions are then applied, with messages being passed from the embedding $\mathbf{V}$ to $\mathbf{C}$ and then back. The result is a bipartite graph with the same topology but new feature matrices. Our policy is then obtained by normalising feature values over all variable nodes and averaging the result into a vector $\mu \in \reals^{4}$. This vector $\mu \in \reals^{4}$ represents the mean of a multivariate normal distribution, \normal{\mu}{\gamma I}, where $\gamma \in \reals$. We note that having the GCNN only output the mean was a design choice to simplify the learning process, and that our design can be extended to also output $\gamma$ or additional distribution information. Any sample from the distribution \normal{\mu}{\gamma I} can be considered an action $a_{0} \in \reals^{4}$, which represents $\{\lambda_{1}, \lambda_{2}, \lambda_{3}, \lambda_{4}\}$ with the non-negativity constraints relaxed. Figure \ref{fig:architecture} provides an overview of this architecture. For a walk-through of the GCNN, see Appendix \ref{sec:forward_pass}.

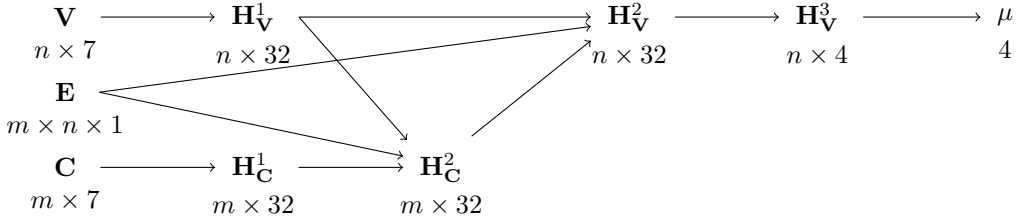
\begin{figure}
\centering
\begin{tikzpicture}
    \draw (2.5,1) node (V) {$\mathbf{V}$};
    \draw (2.5,0) node (E) {$\mathbf{E}$};
    \draw (2.5,-1) node (C) {$\mathbf{C}$};
    \draw ($(V.south)+(0,-0.2)$) node {$n\times 7$};
    \draw ($(E.south)+(0,-0.2)$) node {$m\times n \times 1$};
    \draw ($(C.south)+(0,-0.2)$) node {$m\times 7$};
    
    \draw (5,1) node (V1) {$\mathbf{H}_{\mathbf{V}}^{1}$};
    \draw (5,-1) node (C1) {$\mathbf{H}_{\mathbf{C}}^{1}$};
    \draw ($(V1.south)+(0,-0.2)$) node {$n\times 32$};
    \draw ($(C1.south)+(0,-0.2)$) node {$m\times 32$};
    
    \draw (7.5,-1) node (C2) {$\mathbf{H}_{\mathbf{C}}^{2}$};
    \draw ($(C2.south)+(0,-0.2)$) node {$m\times 32$};
    
    \draw (10,1) node (V2) {$\mathbf{H}_{\mathbf{V}}^{2}$};
    \draw ($(V2.south)+(0,-0.2)$) node {$n\times 32$};
    
    \draw (12.5,1) node (V3) {$\mathbf{H}_{\mathbf{V}}^{3}$};
    \draw ($(V3.south)+(0,-0.2)$) node {$n\times 4$};
    
    \draw (15,1) node (mu) {$\mu$};
    \draw ($(mu.south)+(0,-0.2)$) node {$4$};

    \draw[->] ($(V.east)+(0.2,0)$) -- ($(V1.west)+(-0.1,0)$);
    \draw[->] ($(C.east)+(0.2,0)$) -- ($(C1.west)+(-0.1,0)$);
    
    \draw[->] ($(V1.east)+(0.2,0)$) -- ($(C2.north west)+(-0.05,0.05)$);
    \draw[->] ($(E.east)+(0.2,0)$) -- ($(C2.west)+(-0.1,0.15)$);
    \draw[->] ($(C1.east)+(0.2,0)$) -- ($(C2.west)+(-0.1,0)$);
    
    \draw[->] ($(V1.east)+(0.2,0)$) -- ($(V2.west)+(-0.1,0)$);
    \draw[->] ($(E.east)+(0.2,0)$) -- ($(V2.west)+(-0.1,-0.125)$);
    \draw[->] ($(C2.north east)+(0,0.1)$) -- ($(V2.south west)+(-0.1,0)$);
    
    \draw[->] ($(V2.east)+(0.2,0)$) -- ($(V3.west)+(-0.1,0)$);
    \draw[->] ($(V3.east)+(0.2,0)$) -- ($(mu.west)+(-0.1,0)$);

\end{tikzpicture}
\caption{The architecture of policy network  $\pi_{\theta}(a_{0}|s_{0})$. $\mathbf{H}$ represent hidden layers of the network.}
\label{fig:architecture}
\end{figure}

\subsection{Training Method} \label{subsec:training}

To train our GCNN we use policy gradient methods, specifically the REINFORCE algorithm with baseline and gaussian exploration, see \cite{sutton2018reinforcement} for an overview. An outline of the algorithm is given in Algorithm \ref{alg:reinforce}. 

\begin{algorithm}
\caption{Batch REINFORCE}\label{alg:reinforce}
\DontPrintSemicolon
\SetKwInOut{Input}{Input}
\Input{Policy network $\pi_{\theta}$, MILP instances $batch$, $n_{\text{samples}} \in \naturals$, $N_{r} \in \naturals$}
$\mathcal{L}$ $\gets$ 0 \\
\For{$s_{0}$ in batch}{
    $\mu \gets \pi_{\theta}(\cdot|s_{0})$ \tcp{Note that $\pi_{\theta}(\cdot|s_{0})$ is technically $\normal{\mu}{\gamma I}$}
    \For{i in $\{1, \dots, n_{\text{samples}} \}$}{
        $a_{0} \gets$ sample $\normal{\mu}{\gamma I}$\;
        $s_{1} \gets$ Apply $N_{r}$ rounds of separation and cut selection to $s_{0}$\;
        $r \gets$ Relative dual bound improvement of $s_{1}$ to some baseline \;
        $\mathcal{L}$ $\gets$ $\mathcal{L}$ + ($-r$ $\times$ $log(\pi_{\theta}(a_{0}|s_{0})))$ \tcp{Use log probability for numeric stability}
    }
}
$\theta \gets \theta + \nabla_{\theta}\mathcal{L}$ \tcp{We use the Adam update rule in practice \cite{adam}}
\end{algorithm}

Algorithm \ref{alg:reinforce} is used to update the weights and biases, $\theta$, of our GCNN, $\pi_{\theta}(\cdot|s_{0} \in \mathcal{G})$. It does this for a batch of instances by minimising $\mathcal{L}$, referred to as the loss function, see \cite{goodfellow2016deep}. We used default parameter settings in the Adam update rule, aside from a learning rate with value $5 \times 10^{-4}$. Our training approach is performed offline, and only the final GCNN is used for evaluation. 

\section{Experiments}
\label{sec:experiments}

We use MIPLIB 2017\footnote{MIPLIB 2017 -- The Mixed Integer Programming Library \url{https://miplib.zib.de/}.} \cite{miplib} as our first data set, which we simply refer to it as MIPLIB, and a set of neural network verification instances\footnote{\url{https://github.com/deepmind/deepmind-research/tree/master/neural_mip_solving}} \cite{nair2020solving} as our second data set, which we refer to as NN-Verification. For all subsections we run experiments on instances that have gone through SCIP's default presolve, see \cite{achterberg2020presolve} for an overview on presolve techniques. Each individual run on a presolved instance consists of a single round of presolve (to remove fixed variables), then solving the root node, using 50 separation rounds with a limit of 10 cuts per round. Propagation, heuristics, and restarts are disabled for the runs, with a slightly modified version of SCIP's cut selector in Algorithm \ref{alg:cutsel} being used, where $\boldsymbol{\lambda}$ is defined by the user for each run. A pre-loaded MIP start is also provided, which is the best solution found within 600s when solved with default settings. In the case of less than 10 cuts being selected due to parallelism filtering, the highest filtered scoring cuts are added until the 10 cut per round limit is reached, or no more cuts exist. We believe these conditions best represent a sandbox environment that allows cut selection to be the largest influence on solver performance. Additionally, all results are obtained by averaging results over the SCIP random seeds $\{1, 2, 3\}$. All code for reproducing experiments can be found at \url{https://github.com/Opt-Mucca/Adaptive-Cutsel-MILP}. 

\begin{table}[h]
\centering
\resizebox{\columnwidth}{!}{%
\begin{tabular}{lll}
Criteria & \% MIPLIB & \% NN-Verification \\
\hline
Tags: \textit{feasibility}, \textit{numerics}, \textit{infeasible}, \textit{no solution} & 4.5\%, 17.5\%, 2.8\%, 0.9\% & - \\
Unbounded objective, MIPLIB solution unavailable & 0.9\%, 2.6\% & - \\
Presolve longer than 300s under default conditions & 3.6\% & 0\% \\
No feasible solution found in 600s under default conditions & 10.9\% & 0.2\% \\
Solved to optimality at root & 13.7\% & 0\% \\
Root solve longer than 20s & 22.5\% & 7.1\% \\
Too few cuts applied (< 250) & 7.5\% & 26.4\% \\
Primal-dual difference < 0.5 & 0.9\% & 40.5\% \\
LP errors & 0.5\% & 0.1\% \\
\hline
\end{tabular}
}
\caption{Percentage of instances removed from MIPLIB and NN-Verification data sets.}
\label{tab:instance_filtering}
\end{table}

The modification of SCIP's default cut selector for our experiment is done to standardise the range of the individual cut measures, simplifying the learning process of those measure's coefficients. The measures \texttt{isp} and \texttt{obp} for any cut are in the range $[0,1]$, while the measures \texttt{eff} and \texttt{dcd}, following the assumption that \lpoptimal is separated, are in the range $[0, \infty)$. We therefore substitute \texttt{eff} and \texttt{dcd} in the default SCIP cut scoring rule by the following normalised measures \texttt{eff'} and \texttt{dcd'}:
\begin{align}
    \effs{\coefficients}{\lpoptimal}{\cuts} := \big( \frac{\texttt{log}(\eff{\coefficients}{\lpoptimal} + 1)}{\texttt{log}(\texttt{max}_{\boldsymbol{\alpha'} \in \cuts}\{\eff{\boldsymbol{\alpha'}}{\lpoptimal}\} + 1)} \big)^{2} \label{eq:eff_rule_new}\\
    \dcds{\coefficients}{\lpoptimal}{\incumbent}{\cuts} := \big( \frac{\texttt{log}(\dcd{\coefficients}{\lpoptimal}{\incumbent} + 1)}{\texttt{log}(\texttt{max}_{\boldsymbol{\alpha'} \in \cuts}\{\dcd{\boldsymbol{\alpha'}}{\lpoptimal}{\incumbent}\} + 1)} \big)^{2} \label{eq:dcd_rule_new}
\end{align}

For all experiments SCIP 8.0.1 \cite{scip8} is used, with PySCIPOpt \cite{pyscipopt} as the API, and Gurobi 9.5.1 \cite{gurobi} as the LP solver. PyTorch 1.7.0 \cite{pytorch} and PyTorch-Geometric 2.0.1 \cite{pytorchgeometric} are used to model the GCNN. All experiments for MIPLIB are run on a cluster equipped with Intel Xeon E5-2670 v2 CPUs with 2.50GHz and 128GB main memory, and for NN-Verification on a cluster equipped with Intel Xeon E5-2690 v4 CPUs with 2.60GHz and 128GB main memory.

For instance selection we discard instances from both instance sets that satisfy any of the criteria in Table \ref{tab:instance_filtering}. To minimise bias, instances were discarded if any criteria were triggered in an individual run on any seed under default condition or those tested in Experiment \ref{subsec:lb_experiment}. We believe that these conditions focus on instances where a good selection strategy of cuts can improve the dual bound in a reasonable amount of time. We note that improving the dual bound is a proxy for overall solver performance, and does not necessarily result in improved solution time. We additionally note that only 1000 randomly selected instances from the NN-Verification data set were used as opposed to the entire data set. All instance sets following instance filtering are split into training-test subsets subject to a 80-20 split. 

\begin{figure}[h]
    \centering
    \includegraphics[width=0.775\textwidth]{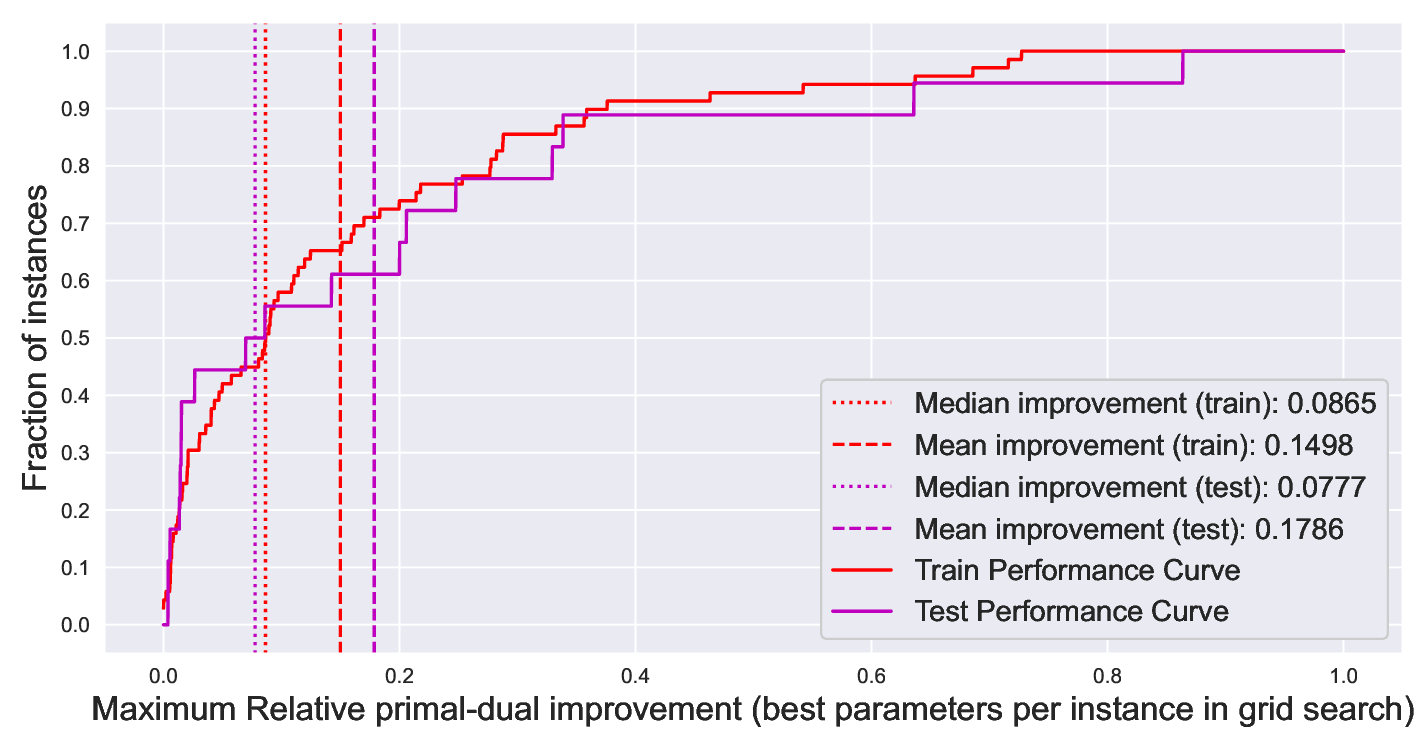} \\
    \includegraphics[width=0.775\textwidth]{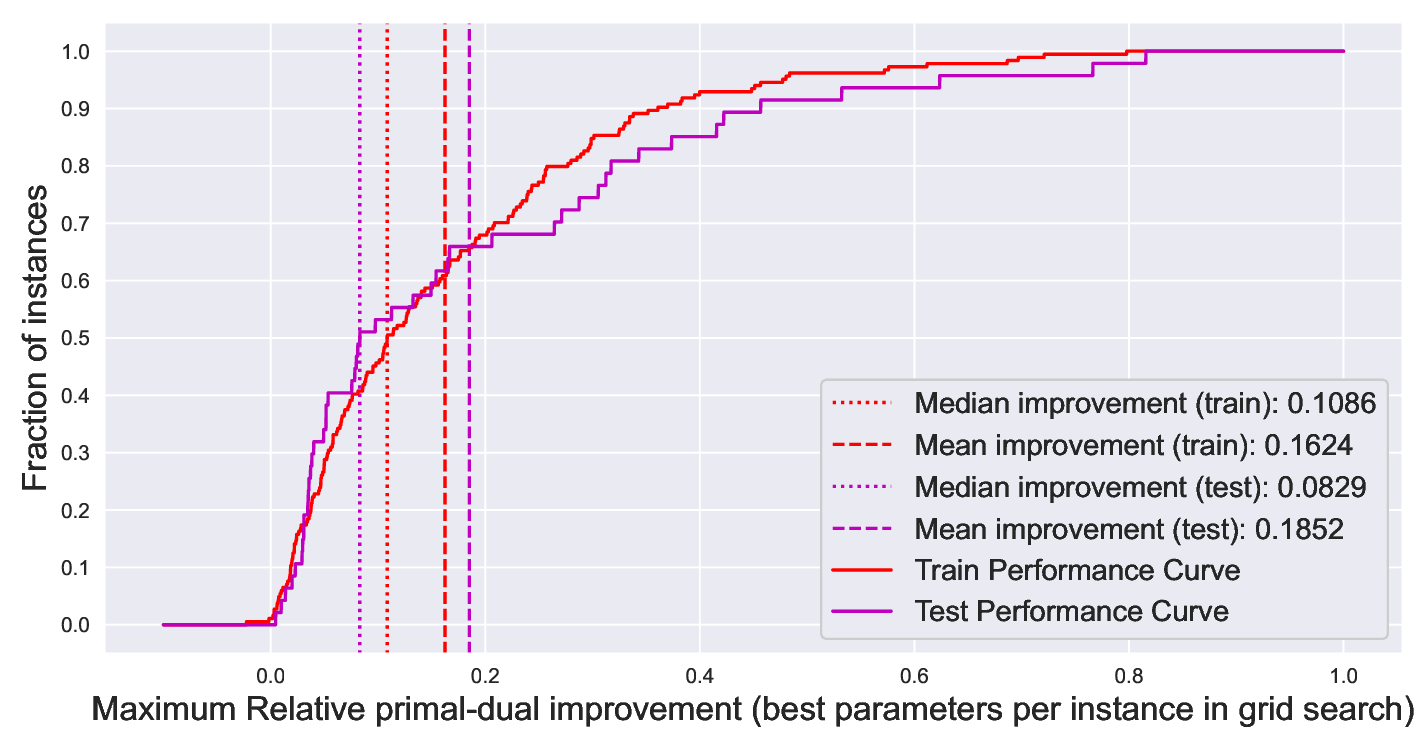}
    \caption{Relative improvement of best choice parameters compared to default parameters in Experiment \ref{subsec:lb_experiment}. (top) MIPLIB. (bottom) NN-Verification.}
    \label{fig:cumulative_grid_search}
\end{figure}

\subsection{Lower Bounding Potential Improvement}
\label{subsec:lb_experiment}

To begin our experiments, we first perform a grid search to give a lower bound on the potential improvement that adaptive cut selection can provide. We generate all parameter scenarios satisfying the following condition: 
\begin{align*}
    \sum_{i=1}^{4} \lambda_{i} = 1, \;\; \text{where} \;\; \lambda_{i} = \frac{\beta_{i}}{10}, \quad \beta_{i} \in \mathbb{N}, \quad \forall i \in \{1,2,3,4\}
\end{align*}
Recall that $\lambda_{i}$ for all $ i \in \{1,2,3,4\}$ are respectively multipliers of the cut scoring measures normalised directed cutoff distance (\texttt{dcd'}), normalised efficacy (\texttt{eff'}), integer support (\texttt{isp}), and objective parallelism (\texttt{obp}).

We solve the root node for all instances and parameter choices, and store the cut selector parameters that result in the smallest primal-dual difference, as well as their relative primal-dual difference improvement compared to that when using default cut selector parameter values. We remove all instances where the worst case parameter choice compared to the best case parameter choice differ by a relative primal-dual difference performance of less than 0.1\%. Additionally, we remove instances where a quarter or more of the parameter choices result in the identical best performance. These removals are made due to the sparse learning opportunities provided by the instances, as the best case performance is minimally different from the worst, or the best case performance is too common. This results in an additional 2.5\% and 0.2\% of instances being removed for MIPLIB, leaving 87 (8.2\%) instances remaining. For NN-Verification no additional instances are removed under these conditions, leaving 231 instances (23.1\%) instances remaining.
We note that all criteria for instance removal in Table \ref{tab:instance_filtering} were performed using the grid search results as well as those under default conditions to ensure no bias throughout instance selection. 

We conclude from the results presented in Figure \ref{fig:cumulative_grid_search} that there exists notable amounts of improvement potential per instance from better cut selection rules. Specifically, we observe that the median relative primal-dual difference improvement compared to standard conditions is at least $7.7\%$ over the training and test sets of both MIPLIB and NN-Verification. We consider this difference very large considering at most 500 cuts (50 rounds of 10 cuts) are added, with this value being only a lower bound on potential improvement as the results come from a grid search of the parameter space. Instance specific results for MIPLIB are available in Appendix \ref{sec:appendix_experiments}.

We draw attention to the aggregated best performing parameter results from the grid search in Table \ref{tab:grid_search}. We see in both data sets that a distance based metric has the largest mean value, being $\lambda_{1}$ (multiplier of \texttt{dcd'}) for MIPLIB and $\lambda_{2}$ (multiplier of \texttt{eff'}) for NN-Verification. We also see that $\lambda_{3}$ and $\lambda_{4}$ take on much larger mean values than those in the default SCIP scoring rule, where they have value 0.1, suggesting that the measures \texttt{isp} and \texttt{obp} are not only useful in distance dominated scoring rules. These are aggregated results, however, and we note that they best summarise how every measure can be useful for some instances, further motivating the potential of instance-dependent based cut selection. We stress that this motivation is also true for the homogeneous NN-Verification, where all parameters are still useful.

\begin{table}[h]
\centering
\begin{tabular}{lccccccc}
& \multicolumn{3}{c}{MIPLIB} && \multicolumn{3}{c}{NN-Verification} \\
\cline{2-4} \cline{6-8}
Parameter & Mean & Median & Std Deviation && Mean & Median & Std Deviation \\
\hline
$\lambda_{1}$ (\texttt{dcd}) & 0.312 & 0.200 & 0.252 && 0.223 & 0.200 & 0.204\\
$\lambda_{2}$ (\texttt{eff}) & 0.177 & 0.100 & 0.181 && 0.315 & 0.300 & 0.216 \\
$\lambda_{3}$ (\texttt{isp}) & 0.232 & 0.200 & 0.226 && 0.220 & 0.200 & 0.214 \\
$\lambda_{4}$ (\texttt{obp}) & 0.279 & 0.200 & 0.248 && 0.241 & 0.200 & 0.218 \\
\hline
\end{tabular}
\caption{Statistics of best choice parameters per instance (train + test) in Experiment \ref{subsec:lb_experiment}.}
\label{tab:grid_search}
\end{table}

\subsection{Random Seed Initialisation} \label{subsec:random_seed}
Let $\theta_{i}$ be the initialised weights and biases using random seed $i$, where $i \in \naturals$. To minimise the bias of our initialised policy with respect to $\boldsymbol{\lambda}=(\lambda_{1}, \lambda_{2}, \lambda_{3}, \lambda_{4})$, the random seed that satisfies \eqref{eq:rand_seed} is used throughout our experiments.
\begin{align}
    \underset{i \in \{0,...,999\}}{\text{argmin}} \sum_{s_{0}} \lVert \expectation [\pi_{\theta_{i}}(\cdot | s_{0})] - [\frac{1}{4},\frac{1}{4},\frac{1}{4},\frac{1}{4}] \rVert_{1} \label{eq:rand_seed}
\end{align}
We believe this random seed minimises bias as the GCNN initially outputs approximately equal values over the data set, allowing the GCNN to best decide the importance of each parameter. This was motivated from the observation that some random initialisations resulted in a cut measure always having an output value of 0 starting from the untrained GCNN. Different random seeds were used for the MIPLIB and NN-Verification experiments, and the random seeds were found using combined training and test sets. 

\begin{figure}[h]
\centering
\includegraphics[width=0.775\textwidth]{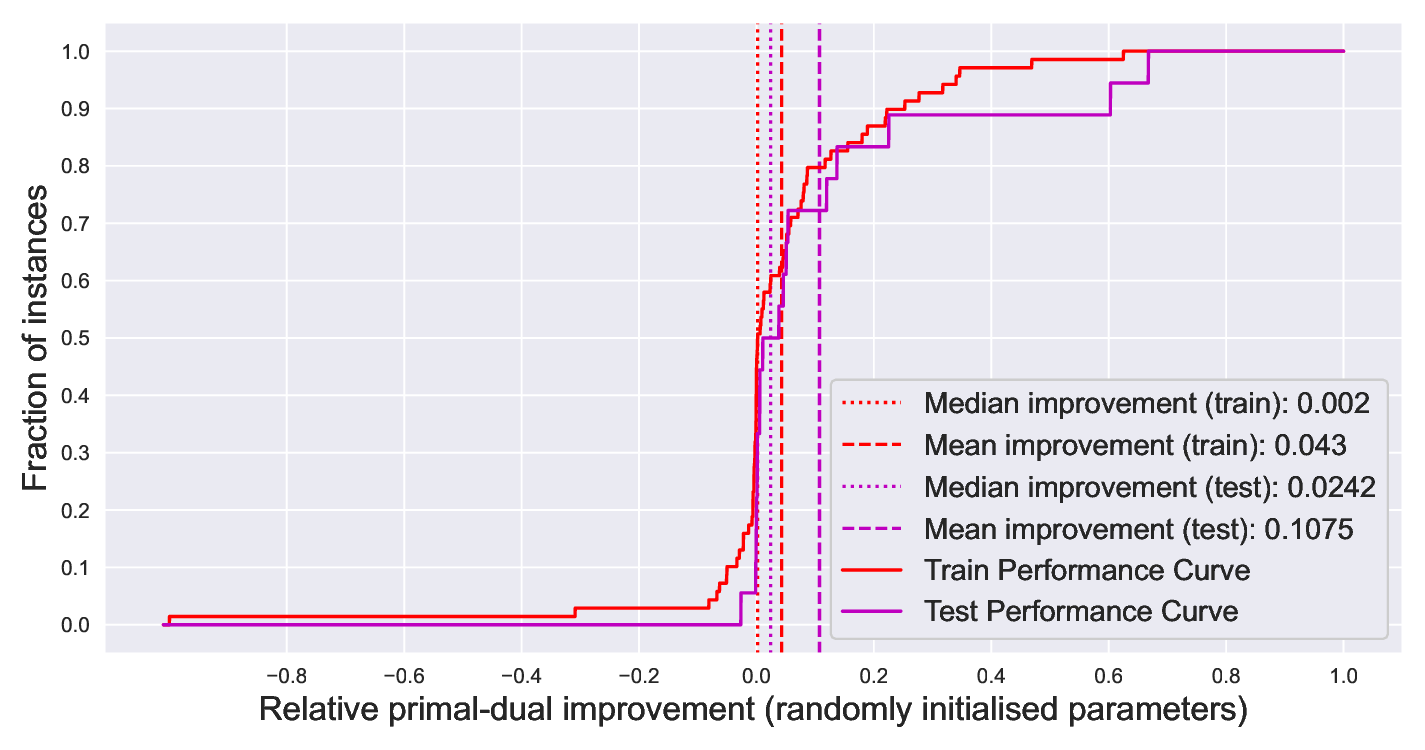} \\
\includegraphics[width=0.775\textwidth]{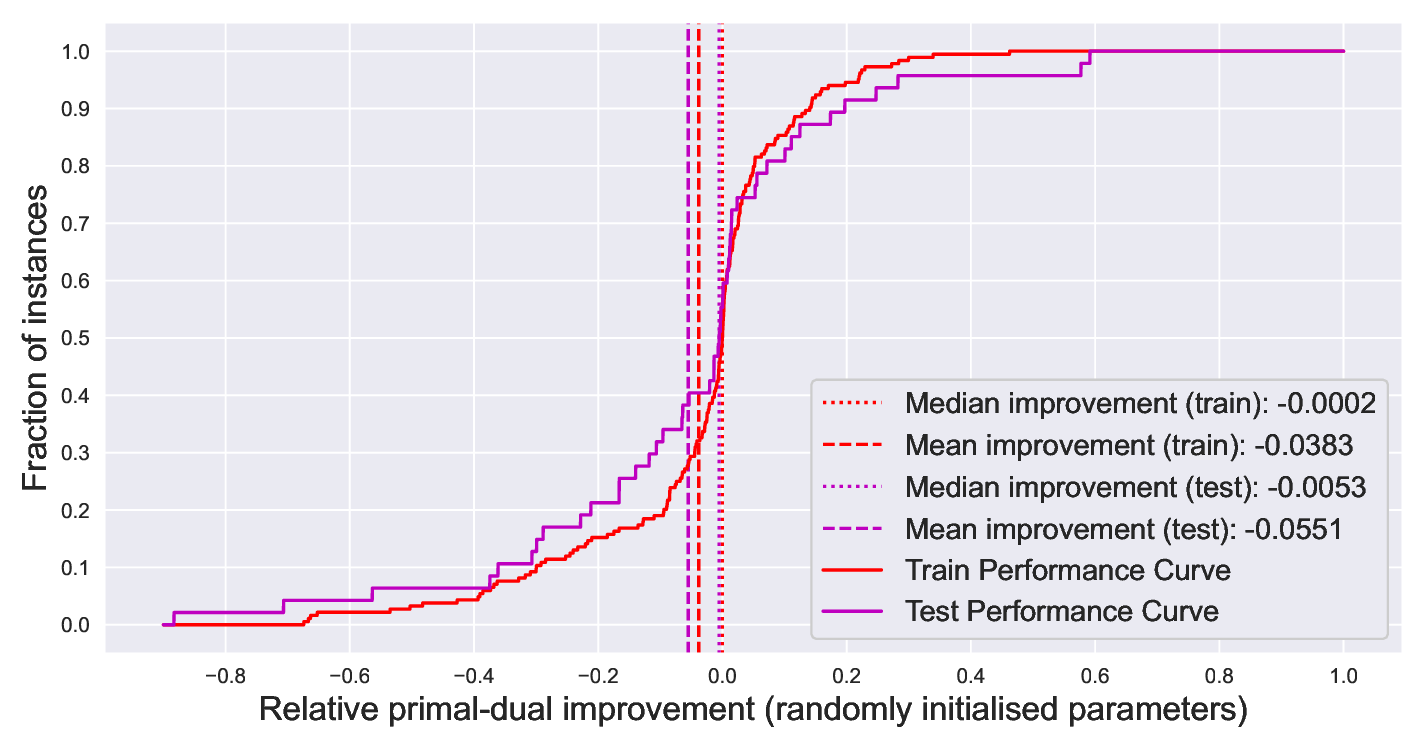}
\caption{Relative primal-dual difference improvement from random initialised parameters compared to default parameters in Experiment \ref{subsec:random_seed}. (top) MIPLIB. (bottom) NN-Verification.}
\label{fig:cumulative_random_seed}
\end{figure}

The performance of the randomly initialised GCNN can be seen in Figure \ref{fig:cumulative_random_seed} and Table \ref{tab:random_seed}, with instance specific results available in Appendix \ref{sec:appendix_experiments}. We observe a larger than expected mean improvement over the test set for MIPLIB, however from the size of the test set and the much lower median improvement, conclude that it's the result of outliers. Surprisingly, the median and mean relative improvement for training and test sets for MIPLIB are positive, while they are negative for NN-Verification. We believe that the positive, albeit small, performance improvement of our random initialisation over default SCIP on MIPLIB is from our slight modification of the cut scoring rule with \texttt{eff'} and \texttt{dcd'}. For NN-verification, we believe the negative performance comes from the decrease in $\lambda_{2}$ (the multiplier of \texttt{eff'}), which is weighted highly in default SCIP, and is important for the instance set according to results in Experiment \ref{subsec:lb_experiment}.

\begin{table}[h]
\centering
\begin{tabular}{lccccccc}
& \multicolumn{3}{c}{MIPLIB} && \multicolumn{3}{c}{NN-Verification} \\
\cline{2-4} \cline{6-8}
Parameter & Mean & Median & Std Deviation && Mean & Median & Std Deviation \\
\hline
$\lambda_{1}$ (\texttt{dcd}) & 0.253 & 0.253 & 0.006 && 0.253 & 0.254 & 0.005\\
$\lambda_{2}$ (\texttt{eff}) & 0.294 & 0.296 & 0.012 && 0.215 & 0.214 & 0.006 \\
$\lambda_{3}$ (\texttt{isp}) & 0.306 & 0.300 & 0.018 && 0.279 & 0.279 & 0.002 \\
$\lambda_{4}$ (\texttt{obp}) & 0.147 & 0.149 & 0.011 && 0.253 & 0.253 & 0.007 \\
\hline
\end{tabular}
\caption{Statistics of random initialised parameters per instance (train + test) in Experiment \ref{subsec:random_seed}.}
\label{tab:random_seed}
\end{table}

\subsection{Standard Learning Method}
\label{subsec:smac}

Before we attempt to determine the capability of our RL framework, policy architecture, and training method, we first design an experiment using SMAC (Sequential Model Algorithm Configuration)\footnote{\url{https://github.com/automl/SMAC3}}, see \cite{smac}. SMAC is a standard package in the field of algorithm configuration, and is largely based on Bayesian optimisation. Unlike our approach, which returns instance-dependent cut selector parameters, SMAC will return a single set of parameter values that works over the entire instance set. It can therefore be thought of as a more intelligent approach than traditional grid searches, which have been used to define SCIP default parameter values. We therefore aim to outperform SMAC given the adaptive advantage of our algorithm. 

We use SMAC4BB, which is targeted at low dimensional and continuous black box functions, and provide SCIP 8.0.1's default values for $\boldsymbol{\lambda}$. We run 250 epochs of SMAC (the same as we will in Experiment \ref{subsec:adaptive_root}), however we note that our approach requires additional solver calls due to taking more than one sample of cut selector parameters from the generated distributions during training. The function that SMAC attempts to minimise is the average primal-dual difference over all instance-seed pairs relative to that produced by SCIP with default cut selector parameter values. 

\begin{figure}[h]
\centering
\includegraphics[width=0.775\textwidth]{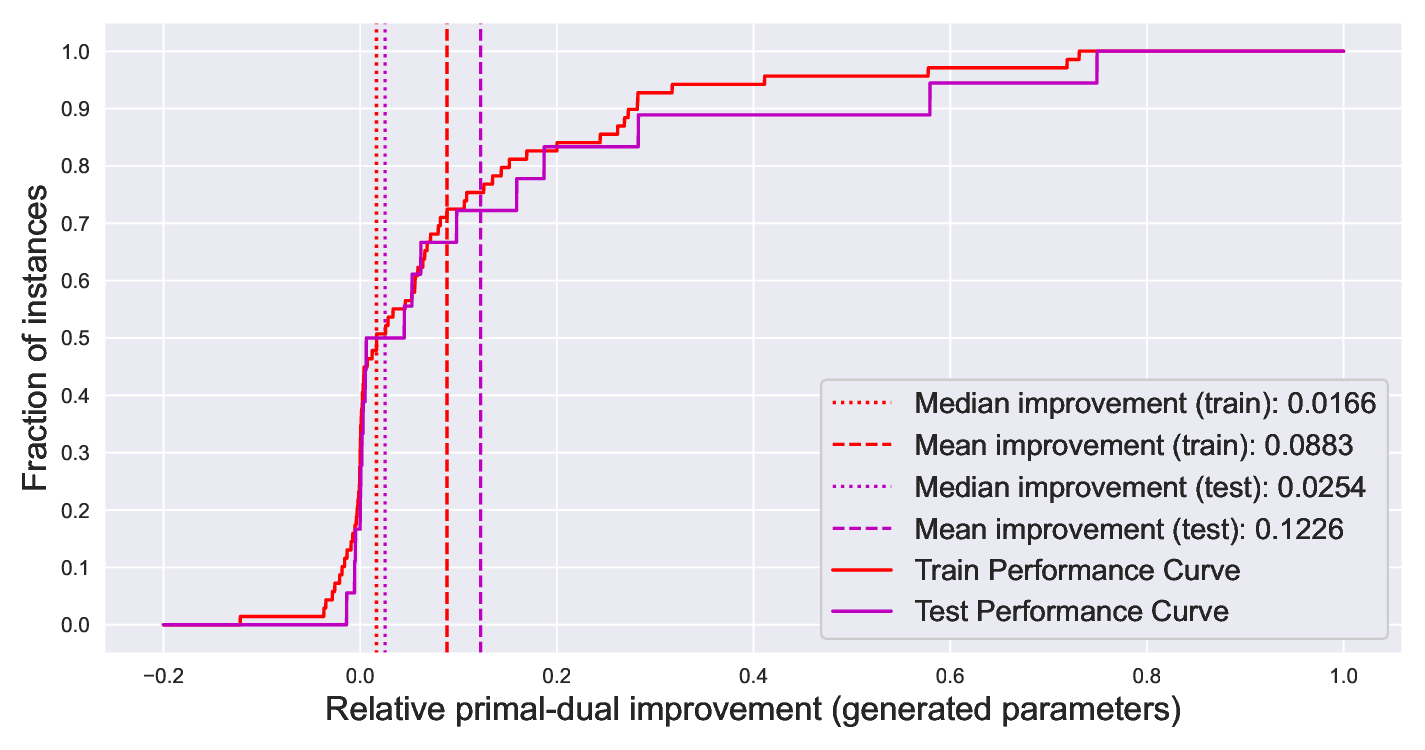} \\
\includegraphics[width=0.775\textwidth]{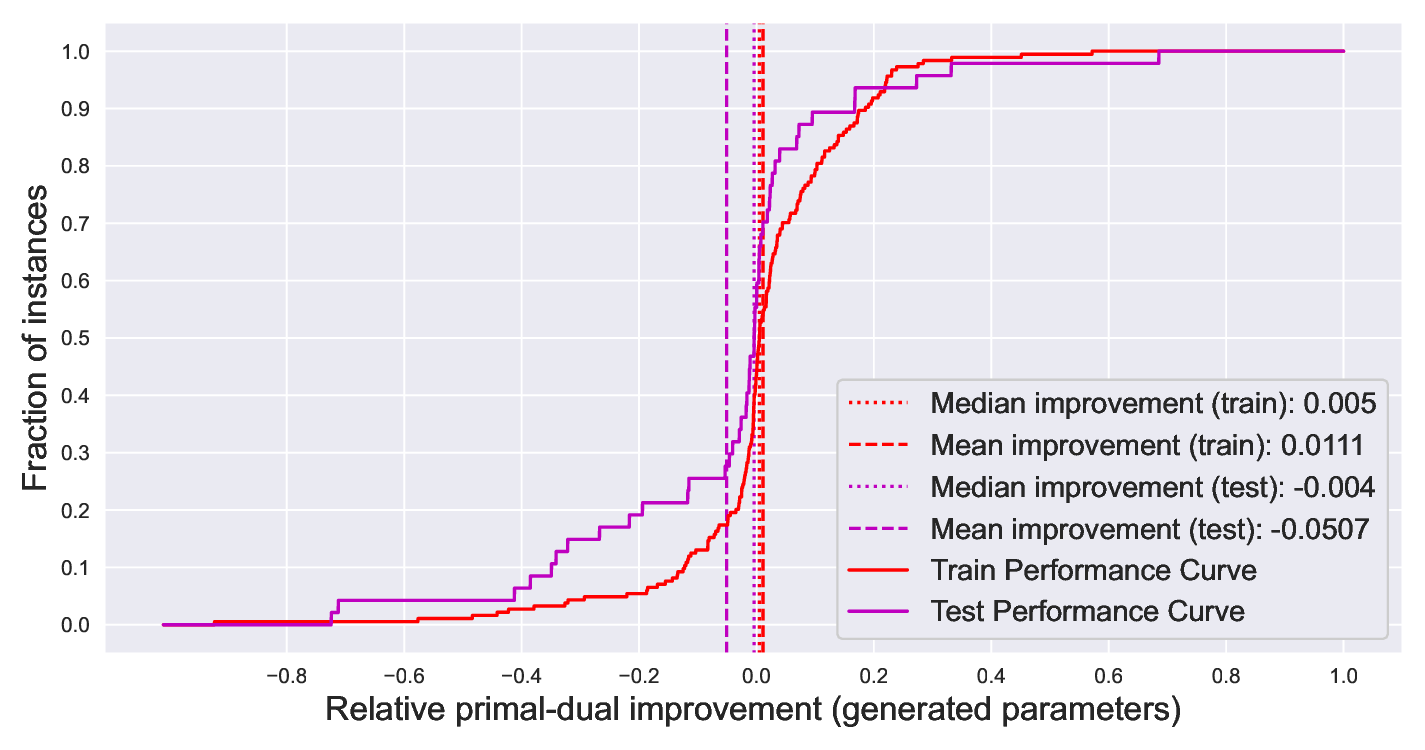}
\caption{Relative primal-dual difference improvement compared to default SCIP from generated parameters in Experiment \ref{subsec:smac}. (top) MIPLIB. (bottom) NN-Verification.}
\label{fig:cumulative_smac}
\end{figure}

We observe an increase in performance over MIPLIB after using SMAC compared to that of the random initialisation as seen in Figure \ref{fig:cumulative_smac}. The median improvement over default SCIP for the training set increases to 1.6\% from 0.2\%, and to 2.5\% from 2.4\% for the test set, with the mean improvement over both sets increasing by at least 2\%. For NN-Verification, we only observe a median increase to 0.5\% from -0.02\% for the training set and -0.4\% from -0.5\% for the test set, with the mean performance of each set increasing by less than 2\%. From the best found constant parameter choices generated by SMAC as displayed in Table \ref{tab:smac}, we conclude that an efficacy dominated cut scoring rule, such as default SCIP, is likely the best choice for NN-Verification if restricted to a non-adaptive rule.

\begin{table}[h]
\centering
\begin{tabular}{lccccc}
& MIPLIB &&&& NN-Verification \\
\cline{2-3} \cline {5-6}
Parameter & Constant &&&& Constant \\
\hline
$\lambda_{1}$ (\texttt{dcd}) & 0.600 &&&& 0.065  \\
$\lambda_{2}$ (\texttt{eff}) & 0.124 &&&& 0.633 \\
$\lambda_{3}$ (\texttt{isp}) & 0.175 &&&& 0.301 \\
$\lambda_{4}$ (\texttt{obp}) & 0.100 &&&& 0.000 \\
\hline
\end{tabular}
\caption{Statistics of generated constant parameters in Experiment \ref{subsec:smac}}
\label{tab:smac}
\end{table}


\subsection{Learning Adaptive Parameters} 
\label{subsec:adaptive_root}

\begin{figure}[h]
\centering
\includegraphics[width=0.775\textwidth]{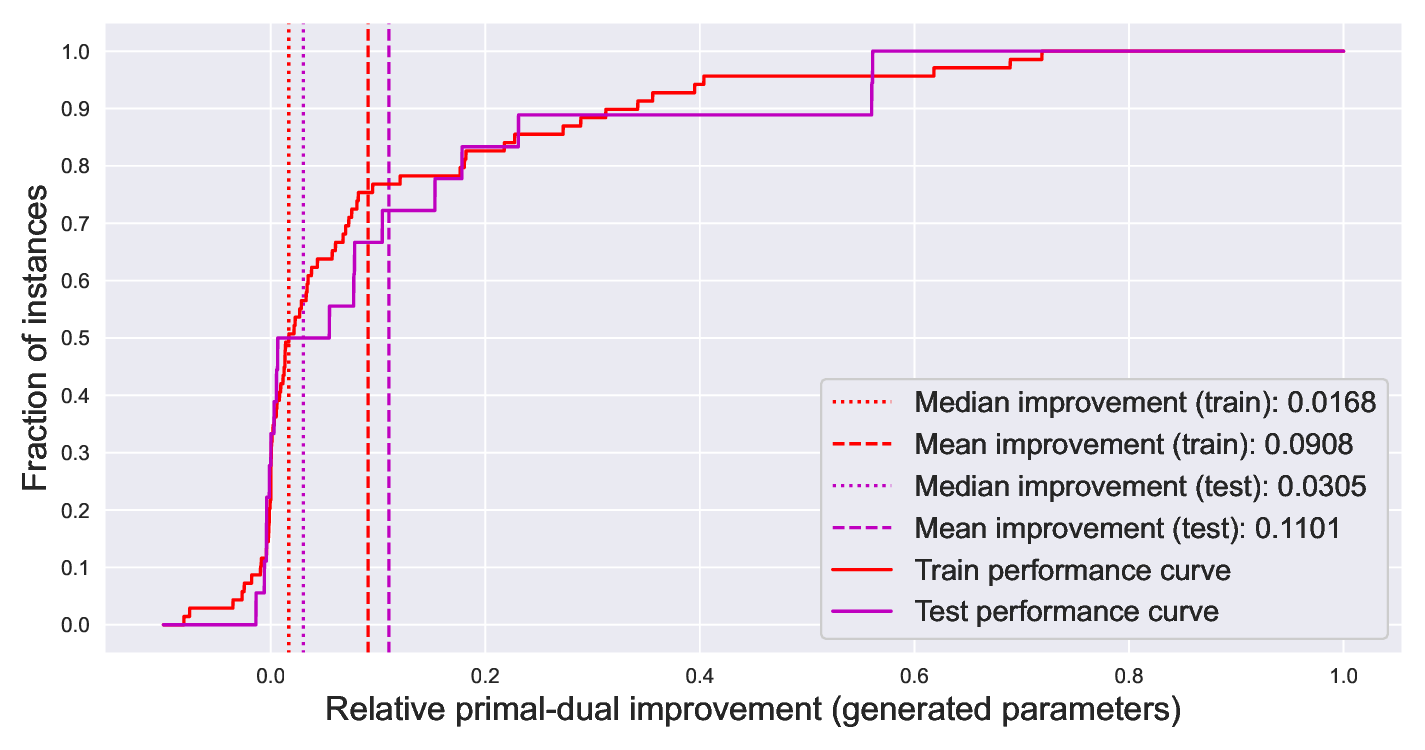} \\
\includegraphics[width=0.775\textwidth]{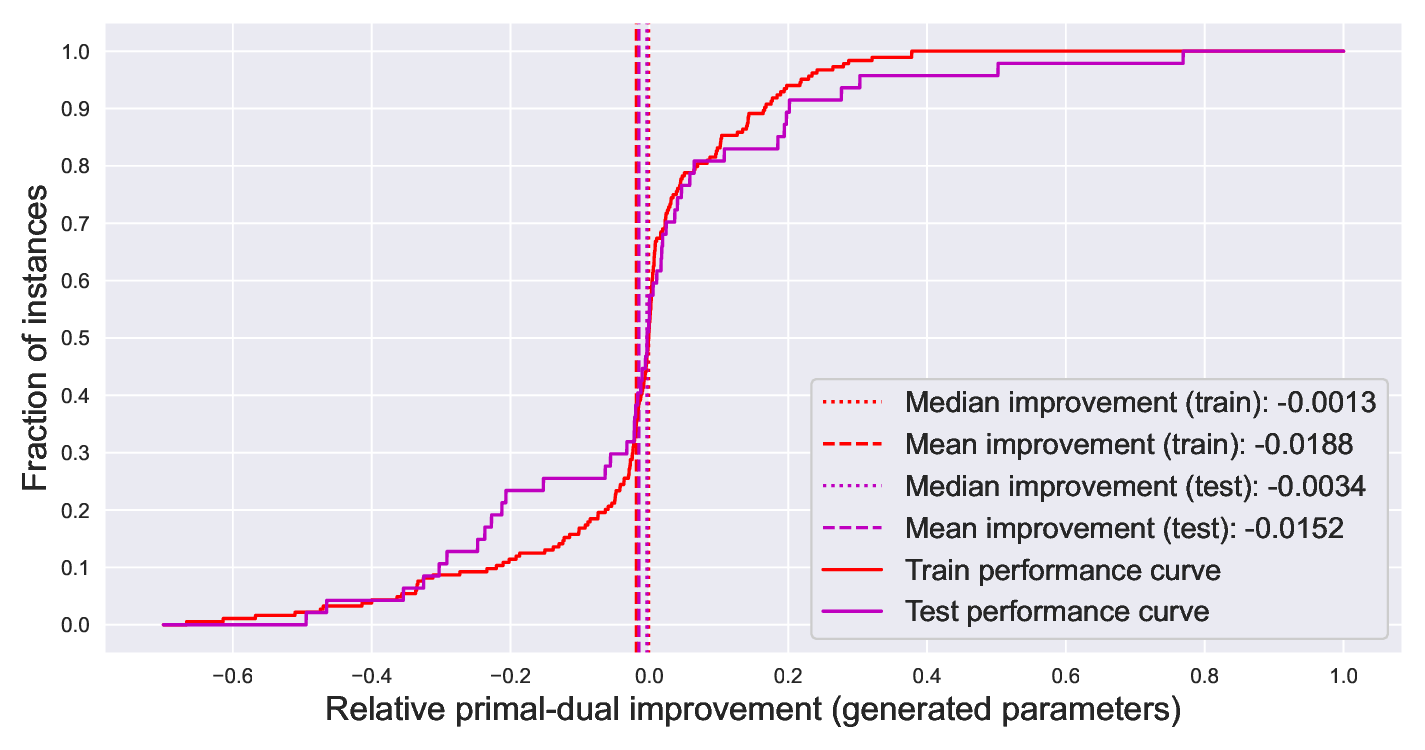}
\caption{Relative primal-dual difference improvement from generated parameters compared to default SCIP parameters in Experiment \ref{subsec:adaptive_root}. (top) MIPLIB. (bottom) NN-Verification.}
\label{fig:cumulative_full}
\end{figure}

We now show the performance of our RL framework, policy architecture, and training method compared to default SCIP parameter choices over MIPLIB and NN-Verification. To do so, we run 5000 iterations of Algorithm \ref{alg:reinforce} (250 epochs), with $n_\text{samples}$ set to 20, and allocate 10\% of instances from the training set per batch. $\gamma$ of the multivariate normal distribution, \normal{\mu}{\gamma I}, is defined by the following, where $n_{\text{epochs}}$ is the total amount of iterations of Algorithm \ref{alg:reinforce} and $i_{\text{epoch}}$ is the current epoch:
\begin{align} \label{eq:gamma}
    \texttt{\gamma}(i_{\text{epoch}}, n_{\text{epochs}}) := 0.01 - \frac{0.009 * i_{\text{epoch}}}{n_{\text{epochs}}}, \;\; i_{\text{epoch}}, n_{\text{epochs}} \in \naturals, \;\; i_{\text{epoch}} \leq n_{\text{epochs}}
\end{align}
We note that $\gamma$ represents one of many opportunities, such as the GCNN structural design and training algorithm, where a substantial amount of additional effort could be invested to (over)tune the learning experiment. We also note that a forward pass of the trained network takes on average less than 0.1s over both data sets, see Table \ref{tab:inference_time} in Appendix \ref{sec:appendix_experiments}, and that updating the GCNN is negligible w.r.t. time compared to solving the MILPs.

The randomly initialised GCNN over the training set of MIPLIB has a median relative primal-dual difference improvement of $0.2\%$ over default as seen in Figure \ref{fig:cumulative_random_seed}, compared to the $1.7\%$ of our MIPLIB trained GCNN as seen in Figure \ref{fig:cumulative_full}. This improvement is minimally better than the $1.6\%$ improvement over default from SMAC in Figure \ref{fig:cumulative_smac}, with our approach slightly improving over SMAC for the training set of MIPLIB, and performing comparably over the test set, having better median improvement and worse mean. These results suggest that our approach works, in that it is comparable with other standard approaches, and can provide improvement over default parameter choices, but that it is unable to capture the full extent of performance improvement that is shown to exist in Experiment \ref{subsec:lb_experiment}. Interestingly, we note that over MIPLIB, Experiments \ref{subsec:lb_experiment} - \ref{subsec:adaptive_root} all on average set $\lambda_{1}$ (the multiplier for \texttt{dcd}) to be the largest coefficient, as seen in Tables \ref{tab:grid_search}, \ref{tab:smac}, and \ref{tab:full}. This is in contrast to the default parameter values used in SCIP 8.0, where the multiplier is set to 0. We believe this difference is strengthened by our computational setup, where we provide SCIP a good initial starting solution. This starting solution is often optimal and better than what initial heuristics would produce. Additionally, as we only add cuts to the root node, the distance to the cut in the direction of the primal solution will reliably point inside the feasible region. This is not the case when the search space has been partitioned like in branch and bound. We also note that over all experiments and learning techniques, for the heterogeneous data set MIPLIB, every cut measure is useful for some instances. For specific instance results, see Appendix \ref{sec:appendix_experiments}.

For NN-Verification, we see from the data presented in Figure \ref{fig:cumulative_full} that our framework, similar to SMAC, failed to perform on the homogeneous data set and capture the performance improvement that was shown to exist in Experiment \ref{subsec:lb_experiment}. We performed comparably to the random initialisation from which training began, and converged to an efficacy dominated scoring rule featuring integer support, see Table \ref{tab:full}, in a near identical manner to the constant scoring rule learned by SMAC. Interestingly, we see that the standard deviation for all measures is near 0, meaning that our framework converged to a constant output. We believe that this is due to a local minima existing for an efficacy dominated model, with default parameters being quite good, and our restriction to static based features, i.e. those available before the first LP solve, being insufficiently diverse for the more homogeneous NN-Verification.

\begin{table}[h]
\centering
\begin{tabular}{lccccccc}
& \multicolumn{3}{c}{MIPLIB} && \multicolumn{3}{c}{NN-Verification} \\
\cline{2-4} \cline{6-8}
Parameter & Mean & Median & Std Deviation && Mean & Median & Std Deviation \\
\hline
$\lambda_{1}$ (\texttt{dcd}) & 0.414 & 0.385 & 0.094 && 0.000 & 0.000 & 0.000\\
$\lambda_{2}$ (\texttt{eff}) & 0.171 & 0.157 & 0.058 && 0.592 & 0.591 & 0.002 \\
$\lambda_{3}$ (\texttt{isp}) & 0.232 & 0.244 & 0.052 && 0.408 & 0.408 & 0.002 \\
$\lambda_{4}$ (\texttt{obp}) & 0.183 & 0.211 & 0.072 && 0.000 & 0.000 & 0.000 \\
\hline
\end{tabular}
\caption{Statistics of generated parameters per instance (train + test) in Experiment \ref{subsec:adaptive_root}}
\label{tab:full}
\end{table}

\subsection{Generalisation to Branch and Bound}
\label{subsec:tree_experiment}

Until this point we have focused on root node restricted experiments and used the primal-dual difference as a surrogate for solver performance. We now deploy the best instance-dependent parameter values from the grid search in Experiment \ref{subsec:lb_experiment} and from our approach in Experiment \ref{subsec:adaptive_root} to the full solving process. We keep the same sandbox environment that we have used until this point, however we no longer limit ourselves to the root node, and we set a time limit of 7200s. 

\begin{table}[h]
\begin{subtable}{1\textwidth}
\centering
\begin{tabular}{lccccccc}
& \multicolumn{3}{c}{MIPLIB} && \multicolumn{3}{c}{NN-Verification} \\
\cline{2-4} \cline{6-8}
Metric & Instance-Seeds & \% Wins & \% Ties && Instances & \% Win & \% Ties \\
\hline
Time & 81 & 60.49 & 0.00 && 675 & 56.74 & 0.00 \\
Nodes & 73 & 58.90 & 10.96 && 628 & 56.21 & 1.59 \\
Dual bound & 180 & 80.05 & 1.66 && 18 & 72.22 & 0.00 \\
\hline
\end{tabular}
\caption{Generalisation to branch and bound of instance-dependent parameters from Experiment \ref{subsec:lb_experiment}}
\label{tab:bb_grid_search}
\end{subtable}
\bigskip
\begin{subtable}{1\textwidth}
\centering
\begin{tabular}{lccccccc}
& \multicolumn{3}{c}{MIPLIB} && \multicolumn{3}{c}{NN-Verification} \\
\cline{2-4} \cline{6-8}
Metric & Instance-Seeds & \% Wins & \% Ties && Instances & \% Win & \% Ties \\
\hline
Time & 81 & 53.09 & 0.00 && 673 & 50.37 & 0.00 \\
Nodes & 74 & 41.89 & 12.16 && 632 & 47.78 & 1.74 \\
Dual bound & 180 & 66.11 & 2.22 && 20 & 50.00 & 0.00 \\
\hline
\end{tabular}
\caption{Generalisation to branch and bound of instance-dependent parameters from Experiment \ref{subsec:adaptive_root}.}
\label{tab:bb_adaptive}
\end{subtable}
\caption{Results of generated cut selection parameters compared to SCIP default parameters. \textit{Time} is a comparison of the solution time of a run, \textit{Nodes} the number of nodes, and \textit{Dual bound} the dual bound when the time limit is hit. For \textit{Time}, instance-seed pairs are considered when at most one of the two runs (default parameters and generated parameters) hit the time limit. For \textit{Nodes}, instance-seed pairs that always solved to optimality are considered, and for \textit{Dual bound} instance-seed pairs that always hit the time limit are considered. The columns \textit{Wins} and \textit{Ties} are the percentage of instance-seed pairs for which the generated parameters outperformed or respectively tied with the default parameters under the given metric.}
\label{tab:branch_and_bound}
\end{table}

From the results presented in Table \ref{tab:branch_and_bound}, specifically Table \ref{tab:bb_grid_search}, we see that instance-dependent cut selector parameters that induce good root node performance do generalise to the larger solving process. This follows from the best grid search instance-dependent parameter values from Experiment \ref{subsec:lb_experiment} clearly outperforming the default parameter choice. We thus believe that in general, the primal-dual difference (or gap) after applying cuts is an adequate surrogate of overall solver performance for a given set of parameters. We note, however, that there exists many solution paths where this statement is not true, and many instances where dual bound progression at the root is a poor surrogate. The improvement generalisation was not as clear for our framework as observable in Table \ref{tab:bb_adaptive}, with our framework outperforming both data sets in terms of time, and losing in terms of nodes. Interestingly, over MIPLIB instance-seed pairs that time out, our framework has a better dual bound than default 66.11\% of the time.

\section{Conclusion} \label{sec:conclusion}
We presented a parametric family of MILPs together with infinitely many family-wide valid cuts. We showed for a specific cut selection rule, that any finite grid search of the parameter space will always miss all parameter values, which select integer optimal inducing cuts in an infinite amount of our instances. We then presented a reinforcement learning framework for learning cut selection parameters, and phrased cut selection in MILP as a Markov decision process. By representing MILP instances as a bipartite graph, we used policy gradient methods to train a graph convolutional neural network. 

The framework generates good performing, albeit sub-optimal, parameter values for a modified variant of SCIP's default cut scoring rule over MIPLIB 2017, with the performance being comparable to standard learning techniques, and clearly better than the random initialisation. Our framework, however, was subject to mode collapse over the NN-Verification data set, and failed to generate a diverse and well performing set of instance-dependent cut selector parameter values.

Results from our grid search experiments showed that there is a large amount of potential improvements to be made in adaptive cut selection, with a median relative primal-dual difference improvement of 7.77\% over MIPLIB and 8.29\% over NN-Verification with only 50 rounds of 10 cuts. The generalisation of these best performing instance-dependent parameter values to branch and bound then revealed a correlation between the primal-dual difference after cut rounds and overall solver performance in terms of both solution time and number of nodes. 

We suggest three key areas of further research for those wanting to build on this research. Firstly, there is a dire need for more instance sets that are sufficiently diverse, non-trivial, yet not overly difficult. Secondly, throughout this paper we restricted ourselves to individual cut measures already featured in SCIP's default rule. Further research could explore rules containing non-linear combinations of additional measures. Third and finally, we suggest that a focus on the larger selection algorithm could lead to further improved performance. For all experiments the separator algorithm's parameters were set to constant values, and we ignored other cut selector related parameters, and restricted ourselves to parallelism based filtering method. We end by noting that a major contribution of this work, the new cut selector plugin for SCIP, enables the last two key areas of further research via easy inclusion of custom cut selection algorithms in a modern MILP solver.

\section*{Acknowledgements} \label{sec:funding}
The work for this article has been conducted in the Research Campus MODAL funded by the German Federal Ministry of Education and Research (BMBF) (fund numbers 05M14ZAM, 05M20ZBM). The described research activities are funded by the Federal Ministry for Economic Affairs and Energy within the project UNSEEN (ID: 03EI1004-C).

\bibliographystyle{unsrt}
\bibliography{mybib}

\newpage
\appendix

\section{Proof of Theorem \ref{thm:main} from Section \ref{sec:theorem}}
\label{sec:proof}

For the following theorem, we will simulate a pure cutting plane approach to solving MILPs using scoring rule \eqref{eq:simple_cut_rule}. We will use custom MILPs, cutting planes, and select exactly one cut per round. Each call to the selection subroutine is called an \textit{iteration} or \textit{round}. The theorem is intended to show how a fixed cut selection rule can consistently choose ``bad'' cuts. 

\maintheorem*

The parametric MILP we use to represent our infinite family of instances is defined as follows, where $a \in \positivereals$ and $d \in [0,1]$:
\begin{align*}
\begin{split}
    \mip := \begin{cases}
    \quad \text{min} & \quad x_{1} - (10+d)x_{2} - ax_{3} \\
    & -\frac{1}{2}x_{2} + 3x_{3} \leq 0 \\
    & -x_{3} \leq 0 \\
    \quad \text{s.t.} & -\frac{1}{2}x_{1} + \frac{1}{2}x_{2} - \frac{7}{2}x_{3} \leq 0 \\
    & \frac{1}{2}x_{1} + \frac{3}{2}x_{3} \leq \frac{1}{2} \\
    & x_1 \in \ints, \quad x_2 \in \reals, \quad x_3 \in \{0,1\}
    \end{cases}
\end{split}
\end{align*}

The polytope of our MILPs LP relaxation is the convex hull of the following points:
\begin{align}
    \mathcal{X} := \{ (0,0,0), (1,0,0), (1,1,0), (\frac{-1}{2}, 3, \frac{1}{2}) \}
\end{align}

The convex hull of $\mathcal{X}$ is a 3-simplex, or alternatively a tetrahedron, see Figure \ref{fig:feasible_region} for a visualisation. For such a feasible region, we can exhaustively write out all integer feasible solutions:

\begin{figure}[h]
    \centering
    \includegraphics[scale=0.75]{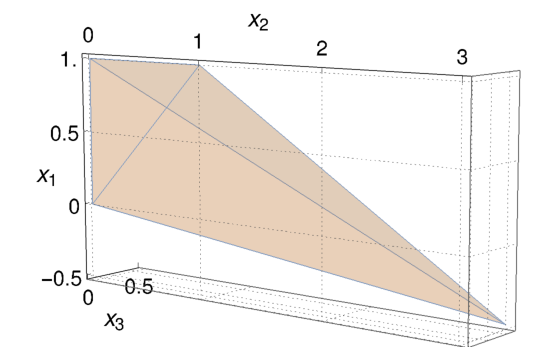}
    \caption{The feasible region of the LP relaxation of \mip.}
    \label{fig:feasible_region}
\end{figure}

\begin{lemma}\label{lem:int_feas_points}
The integer feasible set of \mip, $(a,d) \in \positivereals \times [0,1]$ is:
\begin{align*}
    \{ (0,0,0) \} \cup \{(1,x_2,0) :\forall x_2 \in [0,1]\}
\end{align*}
\end{lemma}
As we are dealing with linear constraints and objectives, we know that $(1,x_2,0)$, where $0<x_2<1$ cannot be optimal without both $(1,1,0)$ and $(1,0,0)$ being also optimal. We therefore simplify the integer feasible set, \integerpoints, to:
\begin{align}
    \integerpoints := \{ (0,0,0), (1,0,0), (1,1,0)\}
\end{align}

At each iteration of adding cuts we will always present exactly three candidate cuts. We name these cuts as follows:
\begin{itemize}
    \item The `good cut', denoted \goodcut: Applying this cut immediately results in the next LP solution being integer optimal.
    \item The `integer support cut', denoted \intsupcut{n}: Applying this cut will result in a new LP solution barely better than the previous iteration. The cut has very high integer support as the name suggests, and would be selected if $\lambda$ from \eqref{eq:simple_cut_rule} is set to a high value. The superscript $n$ refers to the iteration number.
    \item The `objective parallelism cut', denoted \objparalcut{n}: Applying this cut will also result in a new LP solution barely better than the previous iteration. The cut has very high objective parallelism, and would be selected if $\lambda$ from \eqref{eq:simple_cut_rule} is set to a low value. The superscript $n$ refers to the iteration number.
\end{itemize} 

The cuts are defined as follows, where \goodcut has an additional property of it being selected in the case of a scoring tie: 
\begin{align}
    \goodcut:& -10x_1 + 10x_2 + x_3 \leq 0 \label{eq:gc} \\
    \intsupcut{n}:&  -x_1 + x_3 \leq 1 - \epsilon_n \label{eq:isc} \\
    \objparalcut{n}:&  -x_1 + 10x_2 \leq \frac{61}{2} - \epsilon_n \label{eq:opc} 
\end{align}

We use $\epsilon_n$ here to denote a small shift of the cut, with a greater $\epsilon_n$ resulting in a deeper cut. We define $\epsilon_n$ as follows: 
\begin{align}
    \begin{split}
    0< \epsilon_i &< \epsilon_{i+1} \quad \forall i \in \mathbb{N} \\
    \lim_{n\xrightarrow{}\infty} &\epsilon_n = 0.1
    \end{split}
\end{align}

A better overview of the proof of Theorem \ref{thm:main} can now be imagined. At each cut selection round we present three cuts, where for high values of $\lambda$, \intsupcut{n} is selected, and for low values \objparalcut{n} is selected. As the scoring rule \eqref{eq:simple_cut_rule} is linear w.r.t. $\lambda$, our aim is to controllably sandwich the intermediate values of $\lambda$ that will select \goodcut. Specifically, for any given $\lambdadis$, we aim to construct an infinite amount of parameter values for $a$ and $d$ s.t. the intermediate values of $\lambda$ all belong to $\lambdadisgaps$.

\begin{lemma} \label{lem:vertex_set}
The vertex set of the LP relaxation of $\mip, (a,d) \in \reals_{\geq 0} \times [0,1]$, after having individually applied cuts \eqref{eq:gc}, \eqref{eq:isc}, or \eqref{eq:opc} are, respectively: \\
\begin{align}
\goodcutpoints := \integerpoints \cup \{ (\frac{61}{91}, \frac{60}{91}, \frac{10}{91}) \} \\
\intsupcutpoints{n} := \integerpoints \cup \{ (\frac{-1}{2} + \frac{3\epsilon_n}{4}, 3-\frac{3\epsilon_n}{2}, \frac{1}{2} -\frac{\epsilon_n}{4}), (\frac{-1}{2} + \frac{3\epsilon_n}{4}, 3-\epsilon_n, \frac{1}{2} -\frac{\epsilon_n}{4}), (\frac{-1}{2} + \frac{\epsilon_n}{2}, 3-3\epsilon_n, \frac{1}{2} -\frac{\epsilon_n}{2}) \} \\
\objparalcutpoints{n} := \integerpoints \cup \{ (\frac{-1}{2} + \frac{\epsilon_n}{21}, 3 - \frac{2\epsilon_n}{21}, \frac{1}{2} - \frac{\epsilon_n}{63}), (\frac{-1}{2} + \frac{3\epsilon_n}{43}, 3 - \frac{4\epsilon_n}{43}, \frac{1}{2} - \frac{\epsilon_n}{43}), (\frac{-1}{2} + \frac{\epsilon_n}{61}, 3 - \frac{6\epsilon_n}{61}, \frac{1}{2} - \frac{\epsilon_n}{61}) \}
\end{align}
\end{lemma}

\begin{proof}
Apply \goodcut, \intsupcut{n}, and \objparalcut{n} to \mip individually and then compute the vertices of the convex hull of the LP relaxation.
\end{proof}


We note that because both integer support and objective parallelism do not depend on the current LP solution, iteratively applying deeper cuts of the same kind would leave the cut's scores unchanged. Thus, provided they do not separate any integer points and continue to cut off the LP solution, deeper cuts of the same kind can be recursively applied. This is why both \intsupcut{n} and \objparalcut{n} have a superscript. 

After applying a cut of one kind, e.g. \intsupcut{n}, we cannot always simply increment $n$ in the other cut, e.g. \objparalcut{n+1}. This is because \objparalcut{n+1} does not guarantee separation of the now new LP solution in problem $\mip \cap \intsupcut{n}$ for all sequences of $\{\epsilon_n, \epsilon_{n+1}\}$. Instead, we create a variant of \objparalcut{n+1}, namely $\widehat{\objparalcut{n+1}}$, which will always entirely remove the facet of the LP created by adding \intsupcut{n} independent of how the series of $\epsilon_n$ values increase. Once again, we note that as only the RHS values are changing, the score for all cuts within the same type remain unchanged, and thus no two cuts from different types can be applied. We have proven our results using Mathematica \cite{Mathematica}, and a complete notebook containing step-by-step instructions can be found at \url{https://github.com/Opt-Mucca/Adaptive-Cutsel-MILP}. Below we will outline the necessary cumulative lemmas to prove Theorem \ref{thm:main}, and summarise the calculations we have taken to achieve each step. 

\begin{lemma} \label{lem:dom_obj_cut}
Having applied the cut \intsupcut{n} to $\mip, (a,d) \in \reals_{\geq 0} \times [0,1]$, a new facet is created. Applying either \goodcut or a deeper variant of \objparalcut{n+1} cuts off that facet. The deeper variant, denoted $\widehat{\objparalcut{n+1}}$, which differs from \objparalcut{n+1} only by the RHS value is defined as:
\begin{align} \label{eq:dom_obj_paral_cut}
   \widehat{\objparalcut{n+1}} : -x_1 + 10x_2 \leq \frac{61}{2} - 31\epsilon_n
\end{align}
\end{lemma}

\begin{proof}
One can first verify that the vertex set of the facet is \intsupcutpoints{n} $\setminus$ \integerpoints. One can then find the smallest $\epsilon'$ s.t the following cut is valid for all $\mathbf{x} \in \intsupcutpoints{n} \setminus \integerpoints$:
\begin{align*}
    -x_1 + 10x_2 \leq \frac{61}{2} - \epsilon '
\end{align*}
The statement is valid for all $\epsilon' > \frac{61\epsilon_n}{2}$, and we arbitrarily select $\epsilon' = 31\epsilon_n$. One can also check that \goodcut dominates \intsupcut{n} by seeing that it separates all vertices of $\intsupcutpoints{n} \setminus \integerpoints$ for all $n \in \mathbb{N}$. Finally, we need to ensure that no integer solution is cut off. We can verify this by checking that every $\mathbf{x} \in \integerpoints$ satisfies \eqref{eq:dom_obj_paral_cut}. This statement holds whenever $\epsilon_n < 0.1$. Therefore $\epsilon' = 31\epsilon_n$ is valid, and we arrive at the cut $\widehat{\objparalcut{n+1}}$.
\end{proof}

\begin{lemma}
Having applied the cut \objparalcut{n} to $\mip, (a,d) \in \reals_{\geq 0} \times [0,1]$, a new facet is created. Applying \intsupcut{n+1} or \goodcut cuts off that facet. \label{lem:dom_int_cut}
\end{lemma}

\begin{proof}
This follows the same structure as the proof of Lemma \ref{lem:dom_obj_cut}. We get that $\epsilon' > \frac{4\epsilon_n}{43}$, and that $\epsilon' = \epsilon_{n+1}$ is valid w.r.t. the integer constraints.
\end{proof}

Using our definition of integer support and objective parallelism in \eqref{eq:intsup_rule}-\eqref{eq:objparal_rule}, we derive the scores for each cut from the simple cut selection scoring rule \eqref{eq:simple_cut_rule}. We let \mipobjective denote the vector of coefficients from the objective of \mip, $(a,d) \in \positivereals \times [0,1]$. The integer support and objective parallelism values of each cut are as follows:
\begin{align}
    \isp{\goodcut} &= \frac{2}{3} \label{eq:intsup_gc} \\
    \isp{\intsupcut{n}} &= 1 \quad \forall n \in \mathbb{N}\\
    \isp{\objparalcut{n}} &= \frac{1}{2} \quad \forall n \in \mathbb{N} \\
    \obp{\goodcut}{\mipobjective} &= \frac{110+a+10d}{\sqrt{201}\sqrt{1+a^2 + (10+d)^2}} \label{eq:obj_paral_gc} \\
    \obp{\intsupcut{n}}{\mipobjective} &= \frac{1+a}{\sqrt{2}\sqrt{1+a^2 + (10+d)^2}} \quad \forall n \in \mathbb{N} \label{eq:obj_paral_isc} \\
    \obp{\objparalcut{n}}{\mipobjective} &= \frac{101+10d}{\sqrt{101}\sqrt{1+a^2 + (10+d)^2}} \quad \forall n \in \mathbb{N} \label{eq:obj_paral_opc}
\end{align}

Using our simplified cut scoring rule as defined in \eqref{eq:simple_cut_rule}, we derive the necessary conditions defining the $\lambda$ values, which assign \goodcut a score at least as large as the other cuts.

\begin{lemma} \label{lem:goodcut_lambdas}
\goodcut is selected and added to \mip, $(a,d) \in \reals_{\geq 0} \times [0,1]$, using scoring rule \eqref{eq:simple_cut_rule} if and only if a $\lambda$ is used that satisfies the following conditions:
{\normalfont
\begin{align}
    \lambda * \isp{\goodcut} + (1-\lambda) * \obp{\goodcut}{\mipobjective} \geq \lambda * \isp{\intsupcut{n}} + (1-\lambda) * \obp{\intsupcut{n}}{\mipobjective} \label{eq:gc_over_isc} \\
    \lambda * \isp{\goodcut} + (1-\lambda) * \obp{\goodcut}{\mipobjective} \geq \lambda * \isp{\objparalcut{n}} + (1-\lambda) * \obp{\objparalcut{n}}{\mipobjective} \label{eq:gc_over_opc}
\end{align}}
\end{lemma}

\begin{proof}
We know that the integer support and objective parallelism do not depend on $\epsilon_n$ as seen in equations \eqref{eq:intsup_gc}-\eqref{eq:obj_paral_opc}. Our cut selector rule also selects exactly one cut per iteration, namely the largest scoring cut. Therefore, whenever $\lambda$ satisfies constraints \eqref{eq:gc_over_isc} and \eqref{eq:gc_over_opc}, \goodcut will be selected over both \objparalcut{n} and \intsupcut{n}, and applied to \mip. If $\lambda$ does not satisfy constraints \eqref{eq:gc_over_isc} and \eqref{eq:gc_over_opc}, then \goodcut is not the largest scoring cut and will not be applied to \mip. 
\end{proof}
The inequalities \eqref{eq:gc_over_isc}-\eqref{eq:gc_over_opc} define the region, \region, which exactly contains all tuples $(a,d,\lambda) \in \positivereals \times [0,1]^{2}$ that result in \goodcut being the best scoring cut. The region, \region is visualised in Figure \ref{fig:regionplot3d}. We define the function $\regionfunc{a}{d}$ for all $(a,d) \in \reals_{\geq 0} \times [0,1]$, which maps any pairing of $(a,d)$ to the set of $\lambda$ values contained in \region for the corresponding fixed $(a,d)$ values. 
\begin{align}
    \regionfuncnoargs : \reals_{\geq 0} \times [0,1] \xrightarrow{} \powerset([0,1]) 
\end{align}
Here \powerset refers to the power set. We are interested in \region as we believe that we can find a continuous function that contains all $(a,d,\lambda) \in \positivereals \times [0,1]^{2}$ pairings, which score all cuts equally. Using this function, we can find for a fixed $d \in [0,1]$, the values of $a \in \positivereals$ that would result in a $\lambda \in [0,1]$ value that scores all cuts equally. By perturbing our value of $a \in \positivereals$, we aim to generate an infinite amount of $\lambda \in [0,1]$ values that score \goodcut the largest. Then, by choosing different values of $d \in [0,1]$ originally, we aim to find such an infinite set of $\lambda$ values that can adaptively lie between any given finite discretisation of $[0,1]$.

\begin{figure}
    \centering
    \includegraphics[scale=0.7]{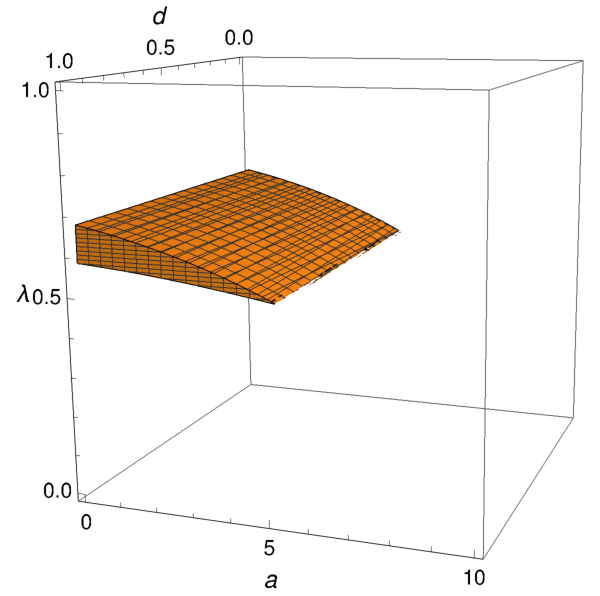}
    \caption{The region, \region, where \goodcut is scored at least as large as both \objparalcut{n} and \intsupcut{n} under cut selection rule \eqref{eq:simple_cut_rule}.}
    \label{fig:regionplot3d}
\end{figure}

\begin{lemma} \label{lem:maxa}
There exists a closed form symbolic solution for the maximum value of $a$ in terms of $d$ over \region. We denote this maximum value of $a$ as a function over $d$, namely $\maxa{d}$, $d \in [0,1]$.
\end{lemma}

\begin{proof}
We verify this by solving the following optimisation problem, where \maxa{d} is printed in Appendix \ref{sec:appendix_mathematica}:
\begin{align*}
    \underset{a, \lambda}{\text{argmax}}\{a \;\; | \;\; (a,d,\lambda) \in \region, \;\; a \geq 0, \;\; 0 \leq d \leq 1, \;\; 0 \leq \lambda \leq 1 \}
\end{align*}
\end{proof}

\begin{lemma}
Closed form symbolic solutions for the upper and lower bounds of $\lambda$ can be found. These bounds are continuous functions defined over $0 \leq d \leq 1$ and $0 \leq a \leq \maxa{d}$, and we refer to them as \lambdaub{a}{d} and \lambdalb{a}{d} respectively. \label{lem:interval_bounds}
\end{lemma}

\begin{proof}
We know that the region respects inequalities \eqref{eq:gc_over_isc} and \eqref{eq:gc_over_opc} and that \maxa{d} is an upper bound on $a$ for all $d \in [0,1]$. Using this information, we can rearrange the inequalities to get \lambdaub{a}{d} and \lambdalb{a}{d}. The result for both \lambdaub{a}{d} and \lambdalb{a}{d} is a ratio of polynomials in terms of the parameters $a$ and $d$. As the zeros of the denominators lie outside of the domains $0 \leq d \leq 1$ and $0 \leq a \leq \maxa{d}$, we can conclude that both \lambdaub{a}{d} and \lambdalb{a}{d} are continuous and defined over our entire domain. These bounds define the interval, $[\lambdalb{a}{d}, \lambdaub{a}{d}]$, of $\lambda$ values for any fixed $a$ and $d$, which result in \goodcut being the largest scoring cut. Taken together with the bounds $0 \leq d \leq 1$ and $0 \leq a \leq \maxa{d}$ they make up \region.
\end{proof}

\begin{lemma} \label{lem:maxa_single_point_region}
The lower and upper bounds for $\lambda$ meet at $a = \maxa{d}$, $d \in [0,1]$. That is, $\lambdaub{\maxa{d}}{d} = \lambdalb{\maxa{d}}{d}$ for all $0 \leq d \leq 1$. This means that for all $d \in [0,1]$, $\lambda=\lambdaub{\maxa{d}}{d}$ (identically $\lambda=\lambdalb{\maxa{d}}{d}$) would score all cuts equally.
\end{lemma}

\begin{proof}
This can be checked by substituting $\maxa{d}$ into the equations of \lambdaub{\maxa{d}}{d} and \lambdalb{\maxa{d}}{d}, equating both sides and rearranging. The result is that $\lambdaub{\maxa{d}}{d} = \lambdalb{\maxa{d}}{d}$.
\end{proof}

\begin{figure}
    \centering
    \includegraphics[scale=0.7]{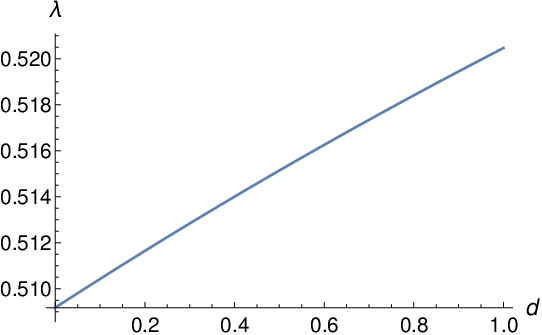}
    \caption{Plot of $\lambda$ values that scores all cuts equally for all $d \in [0,1]$. That is, $\lambdaub{\maxa{d}}{d}$, $d \in [0,1]$ (identically: $\lambdalb{\maxa{d}}{d}$)}
    \label{fig:labda_maxa_d}
\end{figure}

\begin{lemma}
$\lambdaub{\maxa{d}}{d}$ (identically: $\lambdalb{\maxa{d}}{d}$), where $0 \leq d \leq 1$, is a continuous function and has different valued end points. \label{lem:continuous}
\end{lemma}


\begin{proof}
We know from Lemma \ref{lem:interval_bounds} that \lambdaub{a}{d} is continuous, and can conclude that \lambdaub{\maxa{d}}{d} is continuous. The different valued endpoints can be derived by evaluating \lambdaub{\maxa{0}}{0} and \lambdaub{\maxa{1}}{1}, which have the relation \lambdaub{\maxa{1}}{1} $>$ \lambdaub{\maxa{0}}{0}.
\end{proof}

Figure \ref{fig:labda_maxa_d} visualises the function \lambdaub{\maxa{d}}{d} (identically \lambdalb{\maxa{d}}{d}) for $0 \leq d \leq 1$. For any $d \in [0,1]$, these functions alongside slight changes to \maxa{d}, will be used to generate intervals of $\lambda$ values, which score \goodcut the largest and lie between a finite discretisation of $[0,1]$.

\begin{lemma}
\lambdaub{a}{d} - \lambdalb{a}{d} $>$ 0 for all $0 \leq d \leq 1$ and $0 \leq a < \maxa{d}$. That is, $a=\maxa{d}$ is the only time at which \lambdaub{a}{d} = \lambdalb{a}{d} for $0 \leq a \leq \maxa{d}$. \label{lem:unique_maxa}
\end{lemma}

\begin{proof}
We can verify this by setting the constraints \eqref{eq:gc_over_isc}-\eqref{eq:gc_over_opc} to hard equalities, and then solving over the domain $0\leq d \leq 1$ and $0\leq a \leq \maxa{d}$. Solving such a system gives the unique solution $a=\maxa{d}$ for $0 \leq d \leq 1$. As $\lambdaub{0}{d} > \lambdalb{0}{d}$, for all $d \in [0,1]$, and both $\lambdaub{a}{d}$ and $\lambdalb{a}{d}$ are continuous functions from Lemma \ref{lem:interval_bounds}, we can conclude that \lambdaub{a}{d} - \lambdalb{a}{d} $>$ 0 for all $0 \leq d \leq 1$ and $0 \leq a < \maxa{d}$.
\end{proof}

\begin{lemma} \label{lem:interval_construction}
An interval $[\lambdalb{a'}{d}, \lambdaub{a'}{d}] \subseteq \regionfunc{a'}{d}$ can be constructed, where $\lambdaub{a'}{d} - \lambdalb{a'}{d} > 0$ for all $0 \leq d \leq 1$ and $0 \leq a' < \maxa{d}$.
\end{lemma}

\begin{proof}
We define following function, \maxaprime{d}{\epsinterval}, representing \maxa{d} with a shift of \epsinterval:
\begin{align}
    \maxaprime{d}{\epsinterval} := \maxa{d} - \epsinterval, \, \text{where} \quad 0 \leq d \leq 1, \quad 0 < \hat{\epsilon} \leq \maxa{d}
\end{align}
We know from Lemma \ref{lem:unique_maxa} that $a=\maxa{d}$ is the only time at which $\lambdalb{a}{d} = \lambdaub{a}{d}$ for any $d \in [0,1]$. We also know that \lambdaub{a}{d} and \lambdalb{a}{d} are defined over all $0 \leq d \leq 1$ and $0 \leq a \leq \maxa{d}$. Therefore the following holds for any $d \in [0,1]$ and $\epsinterval \in (0, \maxa{d}]$:
\begin{align*}
    \lambdaub{\maxaprime{d}{\epsinterval}}{d} - \lambdalb{\maxaprime{d}{\epsinterval}}{d} > 0
\end{align*}
Additionally, by the definition of \region from the inequalities \eqref{eq:gc_over_isc} - \eqref{eq:gc_over_opc}, we know that the following interval is connected:
\begin{align*}
    \interval{\maxaprime{d}{\epsinterval}}{d} := [\lambdalb{\maxaprime{d}{\epsinterval}}{d}, \lambdaub{\maxaprime{d}{\epsinterval}}{d}], \, \text{where} \quad 0 \leq d \leq 1, \quad 0 < \hat{\epsilon} \leq \maxa{d}
\end{align*}
We therefore can construct a connected non-empty interval \interval{a'}{d} $\subseteq \regionfunc{a'}{d}$ for all $d \in [0,1]$, where $a' = \maxaprime{d}{\epsinterval}$ and $0 \leq \epsinterval < \maxa{d}$. 
\end{proof}

While we have shown the necessary methods to construct an interval of $\lambda$ values, \interval{a}{d}, that result in \goodcut being selected, we have yet to guarantee that at all stages of the solving process, the desired LP optimal solution is taken for all $0 \leq d \leq 1$ and $0 \leq a \leq \maxa{d}$. Specifically, we need to show that the originally optimal point is always $(\frac{-1}{2}, 3, \frac{1}{2})$, that after applying \goodcut the integer solution $(1,1,0)$ is optimal, and that after applying \intsupcut{n} (or \objparalcut{n}) a fractional solution from \intsupcutpoints{n} (or \objparalcutpoints{n}) for all $n \in \naturals$, is optimal. 

\begin{lemma} \label{lem:lp_opt}
The fractional solution $(\frac{-1}{2}, 3, \frac{1}{2})$ is LP optimal for \mip for all $0 \leq d \leq 1$ and $0 \leq a \leq \maxa{d}$.
\end{lemma}

\begin{proof}
This can be done done by substituting all points from $\mathcal{X} \setminus (\frac{-1}{2}, 3, \frac{1}{2})$ into the objective, and then showing that the objective is strictly less when evaluated at $(\frac{-1}{2}, 3, \frac{1}{2})$. This shows that for all $0 \leq d \leq 1$ and $0 \leq a \leq \maxa{d}$ :
\begin{align*}
    \mip \big\rvert_{\mathbf{x}=(\frac{-1}{2}, 3, \frac{1}{2})} < \mip \big\rvert_{\mathbf{x}=\mathbf{x}'} \quad \forall \mathbf{x}' \in \mathcal{X} \setminus \{(\frac{-1}{2}, 3, \frac{1}{2})\}
\end{align*}
\end{proof}

\begin{lemma} \label{lem:int_opt}
The integer solution $(1,1,0)$ is LP optimal after applying \goodcut to \mip for all $0 \leq d \leq 1$ and $0 \leq a \leq \maxa{d}$.
\end{lemma}

\begin{proof}
This can be done in an identical fashion to Lemma \ref{lem:lp_opt}. That is, we show that for all $0 \leq d \leq 1$ and $0 \leq a \leq \maxa{d}$ :
\begin{align*}
    \mip \big\rvert_{\mathbf{x}=(1,1,0)} < \mip \big\rvert_{\mathbf{x}=\mathbf{x}'} \quad \forall \mathbf{x}' \in \goodcutpoints \setminus \{(1,1,0)\}
\end{align*}
\end{proof}

\begin{lemma} \label{lem:intsup_lp_optimal}
Having applied the cut \intsupcut{n} to \mip, a point from \intsupcutpoints{n} $\setminus \integerpoints$ is LP optimal for all $0 \leq d \leq 1$ and $0 \leq a \leq \maxa{d}$.
\end{lemma}

\begin{proof}
This can be done by showing that for any choice of $a \in [0,\maxa{d}]$ and $d \in [0,1]$,
there is at least one point from $\intsupcutpoints{n} \setminus \integerpoints$ at which the objective is strictly less than at all integer points \integerpoints. Specifically, for all $0 \leq d \leq 1$ and $0 \leq a \leq \maxa{d}$:
\begin{align*}
    \exists \mathbf{x}' \in \intsupcutpoints{n} \setminus \integerpoints \quad s.t \quad \mip \big\rvert_{\mathbf{x}=\mathbf{x}'} < \mip \big\rvert_{\mathbf{x}=\mathbf{x}''} \quad \forall \mathbf{x}'' \in \integerpoints
\end{align*}
\end{proof}

\begin{lemma} \label{lem:objparal_lp_optimal}
Having applied the cut \objparalcut{n} to \mip, a point from \objparalcutpoints{n} $\setminus \integerpoints$ is LP optimal for all $0 \leq d \leq 1$ and $0 \leq a \leq \maxa{d}$.
\end{lemma}

\begin{proof}
This proof follows the same logic as that of Lemma \ref{lem:intsup_lp_optimal}.
\end{proof}

We can now prove Theorem \ref{thm:main} using the Lemmas \ref{lem:int_feas_points} - \ref{lem:objparal_lp_optimal} that we have built up throughout this paper.

\maintheorem*

\begin{proof}
From Lemmas \ref{lem:int_feas_points} - \ref{lem:dom_int_cut}, we know the exact vertex set of our feasible region at each stage of the solving process, as well as the exact set of cuts at each round. Furthermore, as at each round only the RHS value for each proposed cut changes, the scoring of the cuts at each new round remains constant, and we can therefore completely describe the three scenarios of how cuts would be added. Let \cutsadded be the set containing all cuts added during the solution process to an instance \mip, where $0 \leq d \leq 1$ and $0 \leq a \leq \maxa{d}$:

\begin{equation}
  \cutsadded := \left \{
  \begin{aligned}
    &\{\intsupcut{n} : \forall n \in \mathbb{N} \}, && \text{if}\ \lambda > \lambdaub{a}{d} \\
    &\{ \goodcut \}, && \text{if}\ \lambdalb{a}{d} \leq \lambda \leq \lambdaub{a}{d}\\
    &\{ \objparalcut{n} : \forall n \in \mathbb{N} \} , && \text{if}\ \lambda < \lambdalb{a}{d}
  \end{aligned} \right.
\end{equation}

From Lemma \ref{lem:goodcut_lambdas} we know the sufficient conditions for a $\lambda$ value that results in \goodcut being scored at least as well as the other cuts. Lemmas \ref{lem:maxa} - \ref{lem:unique_maxa} show how these sufficient conditions can be used to construct the region \region. Moreover, they show that \region is bounded, and that $a=\maxa{d}$, for all $d \in [0,1]$, is the only time at which the following occurs:
\begin{align*}
    \lambdalb{a}{d} = \lambdaub{a}{d} \quad \forall d \in [0,1]
\end{align*}
We therefore conclude that \region is connected. We know from Lemma \ref{lem:continuous} that both \lambdaub{\maxa{d}}{d} and \lambdalb{\maxa{d}}{d} are continuous, where $d \in [0,1]$, and that \lambdaub{\maxa{1}}{1} $>$ \lambdaub{\maxa{0}}{0}. From the intermediate value theorem, we then know the following:
\begin{align}
    \forall \lambda \in [\lambdaub{\maxa{0}}{0},\, \lambdaub{\maxa{1}}{1}], \quad \exists d'\ s.t \ \lambda = \lambdaub{\maxa{d'}}{d'} \label{eq:dprime_exists}
\end{align}
From Lemma \ref{lem:interval_construction} we have shown an explicit way to construct an interval $\interval{a}{d} \subseteq \regionfunc{a}{d}$ for all $(a,d) \in [0,\maxa{d}) \times [0,1]$. We can therefore construct the following intervals:
\begin{align}
    \interval{\maxaprime{d'}{\epsinterval}}{d'} = [\lambdalb{\maxaprime{d'}{\epsinterval}}{d'}, \lambdaub{\maxaprime{d'}{\epsinterval}}{d'}], \, \text{where} \quad 0 \leq d' \leq 1, \quad 0 < \hat{\epsilon} \leq \maxa{d} \label{eq:dprime_interval}
\end{align}
These intervals can be arbitrarily small as \epsinterval can be arbitrarily small. Moreover, as $d'$ values that satisfy \eqref{eq:dprime_exists} can be used, and \lambdalb{a}{d}, \lambdaub{a}{d}, and \maxaprime{d}{\epsinterval} are polynomials, we can generate infinitely many disjoint intervals. We can therefore conclude that for any finite discretisation of $\lambda$, \lambdadis, an interval can be created that contains no values from \lambdadisfull, but contains all values of $\lambda$ for which \mip solves to optimality.
\begin{align}
    \interval{\maxaprime{d'}{\epsinterval}}{d'} \subset \lambdadisgapsfull, \; \text{where} \quad 0 < \epsinterval \leq \maxa{d'}
\end{align}
Finally, Lemmas \ref{lem:lp_opt} - \ref{lem:objparal_lp_optimal} ensure that each stage of the solving process, all cuts are valid for any fractional feasible LP optimal solution for all \mip, where $0 \leq d \leq 1$ and $0 \leq a \leq \maxa{d}$. Moreover, the Lemmas guarantee that only after applying \goodcut is an integer optimal solution found. 

We therefore have shown how fixing a global value of $\lambda$ to a constant for use in the MILP solving process while disregarding all instance information can result in infinitely worse performance for infinitely many instances. 
\end{proof}

\begin{corollary}[thm:main]
There exists an infinite family of MILP instances together with an infinite amount of family-wide valid cuts, which do not solve to integer optimality for any $\lambda$ when using a pure cutting plane approach and applying a single cut per selection round.
\end{corollary}

\begin{proof}
To show this we take the following function, where $0 \leq d \leq 1$:
\begin{align*}
    \reversemaxaprime{d}{\reverseeps} := \maxa{d} + \reverseeps, \quad 0 < \reverseeps \leq 0.1
\end{align*}
Any such value of $a$ retrieved from this function will lie outside of \region for all $0 \leq d \leq 1$. There thus would exist no $\lambda$ value that results in finite termination, as $\goodcut$ is never scored at least as high as the other cuts. 

Similar to the proof of Theorem \ref{thm:main}, we need to ensure that the LP optimal point at all times during the solving process is appropriate, and that the same integer optimal point stays integer optimal for all $0 \leq d \leq 1$ and $\maxa{d} < a \leq \reversemaxaprime{d}{\reverseeps}$. We therefore redo the proofs of Lemmas \ref{lem:lp_opt} - \ref{lem:objparal_lp_optimal} but change the range of values of $a$. 
\end{proof}

\section{Functions of Appendix \ref{sec:proof}}
\label{sec:appendix_mathematica}

$\maxa{d} := \scriptstyle \frac{-2680 \sqrt{101} d+2020 \sqrt{201} d-6767 \sqrt{2}-27068 \sqrt{101}+22220 \sqrt{201}}{6767 \sqrt{2}-202 \sqrt{201}}$

$\lambdalb{a}{d}_{1} := \scriptstyle \sqrt{20301} \sqrt{\left(a^2+d (d+20)+101\right)}$

$\lambdalb{a}{d}_{2} :=  \scriptstyle 101 a^2+a \left(-20 \left(\sqrt{20301}-101\right) d-202 \left(\sqrt{20301}-110\right)\right)$

$\lambdalb{a}{d}_{3} :=  \scriptstyle 20 d \left(-10 \left(\sqrt{20301}-151\right) d-211 \sqrt{20301}+31411\right)-22220 \sqrt{20301}+3272501$

$\lambdalb{a}{d}_{4} := \scriptstyle -606 a^2+12 a \left(10 \left(\sqrt{20301}-101\right) d+101 \left(\sqrt{20301}-110\right)\right)$

$\lambdalb{a}{d}_{5} := \scriptstyle 120 d \left(10 \left(\sqrt{20301}-151\right) d+211 \sqrt{20301}-31411\right)+606 \left(220 \sqrt{20301}-32401\right)$

$\lambdaub{a}{d}_{6} := \scriptstyle 5555 a^2+24 a \left(10 \left(\sqrt{20301}-101\right) d+101 \left(\sqrt{20301}-110\right)\right)$

$\lambdalb{a}{d}_{7} := \scriptstyle d \left(\left(2400 \sqrt{20301}-355633\right) d+50640 \sqrt{20301}-7403300\right)+505 \left(528 \sqrt{20301}-76409\right)$

$\lambdalb{a}{d} := \frac{2\left(\lambdalb{a}{d}_{1} \sqrt{\lambdalb{a}{d}_{2} + \lambdalb{a}{d}_{3}} + \lambdalb{a}{d}_{4} + \lambdalb{a}{d}_{5} \right)}{\lambdalb{a}{d}_{6} + \lambdalb{a}{d}_{7}}$

$\lambdaub{a}{d}_{1} := \scriptstyle -\left(a^2+d (d+20)+101\right)$

$\lambdaub{a}{d}_{2} := \scriptstyle \left(2 \sqrt{402}-203\right) a^2+a \left(20 \left(\sqrt{402}-2\right) d+222 \sqrt{402}-842\right)+20 d \left(-10 d+\sqrt{402}-220\right)+220 \sqrt{402}-24401$

$\lambdaub{a}{d}_{3} := \scriptstyle \left(6 \sqrt{402}-609\right) a^2+6 a \left(10 \left(\sqrt{402}-2\right) d+111 \sqrt{402}-421\right)+60 d \left(-10 d+\sqrt{402}-220\right)+660 \sqrt{402}-73203 $

$\lambdaub{a}{d}_{4} := \scriptstyle \left(6 \sqrt{402}-475\right) a^2+6 a \left(10 \left(\sqrt{402}-2\right) d+111 \sqrt{402}-421\right)+2 d \left(-233 d+30 \sqrt{402}-5260\right)+660 \sqrt{402}-59669$

$\lambdaub{a}{d} := \frac{\sqrt{402}\sqrt{\lambdaub{a}{d}_{1} \lambdaub{a}{d}_{2}} + \lambdaub{a}{d}_{3}}{\lambdaub{a}{d}_{4}}$

\section{A Guide to a Forward Pass of the GCNN}
\label{sec:forward_pass}

This section should be used to provide an intuitive understanding of our policy network, which is parameterised as a GCNN. For a more complete introduction to graph neural networks that also provides helpful visualisations, we refer readers to \cite{gnn}. Throughout this section we will also refer to multi-layer perceptrons, which from now we simply refer to as (feed-forward) neural networks. A neural network is a function, which passes its input through a series of alternating linear transformations and non-linear activation functions. Note that while our design makes use of standard neural networks, other architecture types can be used. We refer readers to \cite{goodfellow2016deep} for a thorough overview.

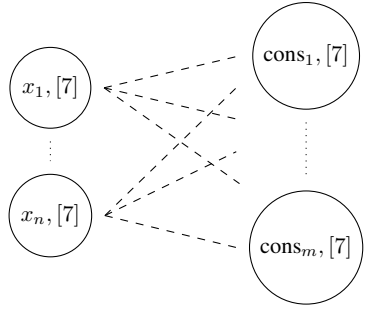
\begin{figure}[h]
\centering
\begin{tikzpicture}[scale=0.85,transform shape]
    \node[draw,circle] (x1) at (9,-2.5) {$x_{1}, [7]$};
    \node[draw,circle] (xn) at (9,-4.5) {$x_{n}, [7]$};
    
    \node[draw,circle] (c1) at (13,-2) {$\text{cons}_{1}, [7]$};
    \node[draw,circle] (cm) at (13,-5) {$\text{cons}_{m}, [7]$};
    
    \draw[dotted] ($(x1.south)+(0, -0.2)$) -- ($(xn.north) + (0, 0.2)$);
    \draw[dotted] ($(c1.south)+(0, -0.2)$) -- ($(cm.north) + (0, 0.2)$);
    
    \draw[dashed] ($(x1.east)+(0.2,0)$) -- ($(c1.west)+(-0.2,0)$);
    \draw[dashed] ($(x1.east)+(0.2,0)$) -- ($(c1.west)+(-0.2,-1)$);
    \draw[dashed] ($(x1.east)+(0.2,0)$) -- ($(c1.west)+(-0.2,-2)$);
    
    \draw[dashed] ($(xn.east)+(0.2,0)$) -- ($(cm.west)+(-0.2,0)$);
    \draw[dashed] ($(xn.east)+(0.2,0)$) -- ($(cm.west)+(-0.2,1.5)$);
    \draw[dashed] ($(xn.east)+(0.2,0)$) -- ($(cm.west)+(-0.2,2.5)$);
\end{tikzpicture}
\caption{A visualisation the initial state $s_{0}$. At each node the size of the corresponding feature vector is given, e.g. [7]. Note that the edges additionally have features, but do not appear for ease of visualisation.}
\label{fig:gcnn_1}
\end{figure}

Recall that we present our MILP instance via a constraint-variable bipartite graph, and that a variable and a constraint share an edge when the variable appears in the constraint with a non-zero coefficient. See Figure \ref{fig:gcnn_1} for an initial representation of $s_{0} \in \mathcal{G}$. Recall also that the goal of the GCNN is to parameterise our policy, $\pi_{\theta}(\cdot|s_{0} \in \mathcal{G})$, outputting the mean, $\mu \in \reals^{4}$, of a distribution, \normal{\mu}{\gamma I}, over the cut selector parameter space.

We will now begin the forward pass of the GCNN. Consider a node of the bipartite graph that represents the variable $x_{i}$ of the MILP. This node has an attached set of features, see Table \ref{tab:features} for a complete list, which form a vector. This feature vector gets transformed by a neural network. In our design this initially transforms our 7-dimensional feature vector to a 32-dimensional vector. This operation gets applied to all feature vectors representing variables, using the same neural network. The new bipartite graph is denoted $\mathbf{H}_{\mathbf{V}}^{1}$. We also do this procedure for all feature vectors corresponding to a constraint of the MILP, albeit with a different neural network, which results in $\mathbf{H}_{\mathbf{C}}^{1}$. The result is visualised in Figure \ref{fig:gcnn_2}. We note that there is no order to the two transformations.

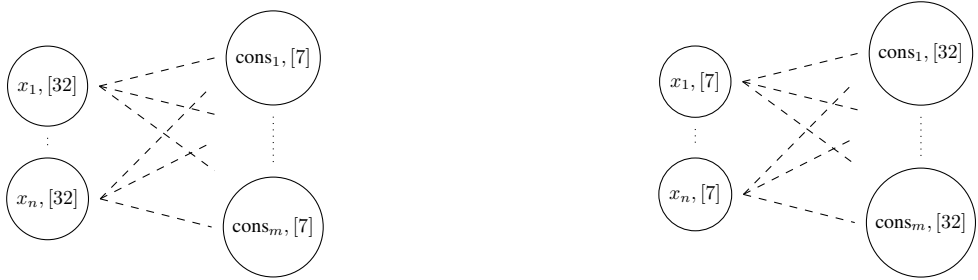
\begin{figure}[h]
\begin{subfigure}[b]{0.475\textwidth}
\centering
\begin{tikzpicture}[scale=0.75,transform shape]
    \node[draw,circle] (x1) at (9,-2.5) {$x_{1}, [32]$};
    \node[draw,circle] (xn) at (9,-4.5) {$x_{n}, [32]$};
    
    \node[draw,circle] (c1) at (13,-2) {$\text{cons}_{1}, [7]$};
    \node[draw,circle] (cm) at (13,-5) {$\text{cons}_{m}, [7]$};
    
    \draw[dotted] ($(x1.south)+(0, -0.2)$) -- ($(xn.north) + (0, 0.2)$);
    \draw[dotted] ($(c1.south)+(0, -0.2)$) -- ($(cm.north) + (0, 0.2)$);
    
    \draw[dashed] ($(x1.east)+(0.2,0)$) -- ($(c1.west)+(-0.2,0)$);
    \draw[dashed] ($(x1.east)+(0.2,0)$) -- ($(c1.west)+(-0.2,-1)$);
    \draw[dashed] ($(x1.east)+(0.2,0)$) -- ($(c1.west)+(-0.2,-2)$);
    
    \draw[dashed] ($(xn.east)+(0.2,0)$) -- ($(cm.west)+(-0.2,0)$);
    \draw[dashed] ($(xn.east)+(0.2,0)$) -- ($(cm.west)+(-0.2,1.5)$);
    \draw[dashed] ($(xn.east)+(0.2,0)$) -- ($(cm.west)+(-0.2,2.5)$);
\end{tikzpicture}
\end{subfigure}
\hfill
\begin{subfigure}[b]{0.475\textwidth}
\centering
\begin{tikzpicture}[scale=0.75,transform shape]
    \node[draw,circle] (x1) at (9,-2.5) {$x_{1}, [7]$};
    \node[draw,circle] (xn) at (9,-4.5) {$x_{n}, [7]$};
    
    \node[draw,circle] (c1) at (13,-2) {$\text{cons}_{1}, [32]$};
    \node[draw,circle] (cm) at (13,-5) {$\text{cons}_{m}, [32]$};
    
    \draw[dotted] ($(x1.south)+(0, -0.2)$) -- ($(xn.north) + (0, 0.2)$);
    \draw[dotted] ($(c1.south)+(0, -0.2)$) -- ($(cm.north) + (0, 0.2)$);
    
    \draw[dashed] ($(x1.east)+(0.2,0)$) -- ($(c1.west)+(-0.2,0)$);
    \draw[dashed] ($(x1.east)+(0.2,0)$) -- ($(c1.west)+(-0.2,-1)$);
    \draw[dashed] ($(x1.east)+(0.2,0)$) -- ($(c1.west)+(-0.2,-2)$);
    
    \draw[dashed] ($(xn.east)+(0.2,0)$) -- ($(cm.west)+(-0.2,0)$);
    \draw[dashed] ($(xn.east)+(0.2,0)$) -- ($(cm.west)+(-0.2,1.5)$);
    \draw[dashed] ($(xn.east)+(0.2,0)$) -- ($(cm.west)+(-0.2,2.5)$);
\end{tikzpicture}
\end{subfigure}
\caption{(Left) $\mathbf{H}_{\mathbf{V}}^{1}$. (Right) $\mathbf{H}_{\mathbf{C}}^{1}$}
\label{fig:gcnn_2}
\end{figure}

The key to graph neural networks, for example the GCNN, is using the same neural network to transform multiple feature vectors, e.g. all constraint feature vectors. This allows the GCNN to take arbitrarily sized bipartite graphs as input, and therefore work on any MILP instance.

Until this point, no information has been shared between any two variables or constraints. In our design information is gathered per node from its neighbours by summing the transformed feature vectors of all incident edges and adjacent nodes. Note that due to the bipartite nature of our graph information either flows from variables to constraints, or vice-versa. This gathering of information for either all variable nodes or all constraint nodes is called a \emph{half-convolution}, or alternatively message passing. To perform this half-convolution we require the transformed feature vectors to all have the same dimension, with 32 being our choice. Note that the feature vectors on the edges are also transformed during the half-convolution. For a feature vector representing a variable, the half-convolution is defined as:
\begin{align*}
    \mathbf{v}_{i}'' := \mathtt{nn'}(\sum_{j \in N(x_{i})}(\mathtt{nn}(\mathbf{v}_{i}' + \mathbf{e}'_{(i,j)} + \mathbf{c}_{j}')), \mathbf{v}_{i}')
\end{align*}
Here, $\mathbf{v}_{i}''$, $\mathbf{v}_{i}'$, $\mathbf{c}_{j}'$, and $\mathbf{e}_{(i,j)}'$ are transformed feature vectors of variable $x_{i}$, constraint $j$, and edge $(i,j)$. The functions $\mathtt{nn}$ and $\mathtt{nn'}$ are neural networks, with $\mathtt{nn}'$ taking a concatenated input, and $N(x_{i})$ is the neighbourhood of the variable node of $x_{i}$. For a feature vector representing a constraint, the half-convolution is defined in a mirrored manner. See Figure \ref{fig:gcnn_3} for $\mathbf{H}_{\mathbf{C}}^{2}$ and $\mathbf{H}_{\mathbf{V}}^{2}$.

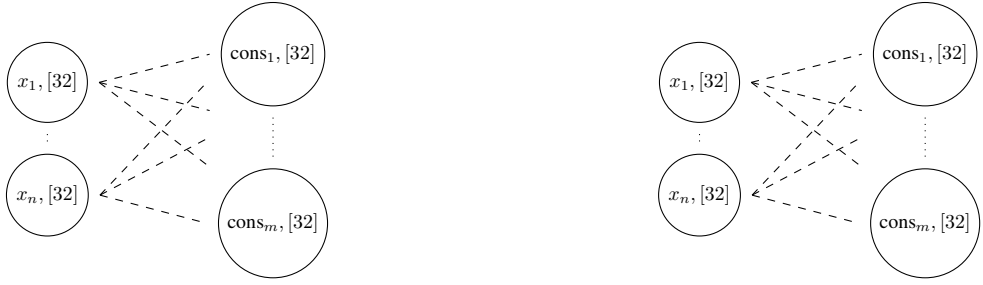
\begin{figure}[h]
\begin{subfigure}[b]{0.475\textwidth}
\centering
\begin{tikzpicture}[scale=0.75,transform shape]
    \node[draw,circle] (x1) at (9,-2.5) {$x_{1}, [32]$};
    \node[draw,circle] (xn) at (9,-4.5) {$x_{n}, [32]$};
    
    \node[draw,circle] (c1) at (13,-2) {$\text{cons}_{1}, [32]$};
    \node[draw,circle] (cm) at (13,-5) {$\text{cons}_{m}, [32]$};
    
    \draw[dotted] ($(x1.south)+(0, -0.2)$) -- ($(xn.north) + (0, 0.2)$);
    \draw[dotted] ($(c1.south)+(0, -0.2)$) -- ($(cm.north) + (0, 0.2)$);
    
    \draw[dashed] ($(x1.east)+(0.2,0)$) -- ($(c1.west)+(-0.2,0)$);
    \draw[dashed] ($(x1.east)+(0.2,0)$) -- ($(c1.west)+(-0.2,-1)$);
    \draw[dashed] ($(x1.east)+(0.2,0)$) -- ($(c1.west)+(-0.2,-2)$);
    
    \draw[dashed] ($(xn.east)+(0.2,0)$) -- ($(cm.west)+(-0.2,0)$);
    \draw[dashed] ($(xn.east)+(0.2,0)$) -- ($(cm.west)+(-0.2,1.5)$);
    \draw[dashed] ($(xn.east)+(0.2,0)$) -- ($(cm.west)+(-0.2,2.5)$);
\end{tikzpicture}
\end{subfigure}
\hfill
\begin{subfigure}[b]{0.475\textwidth}
\centering
\begin{tikzpicture}[scale=0.75,transform shape]
    \node[draw,circle] (x1) at (9,-2.5) {$x_{1}, [32]$};
    \node[draw,circle] (xn) at (9,-4.5) {$x_{n}, [32]$};
    
    \node[draw,circle] (c1) at (13,-2) {$\text{cons}_{1}, [32]$};
    \node[draw,circle] (cm) at (13,-5) {$\text{cons}_{m}, [32]$};
    
    \draw[dotted] ($(x1.south)+(0, -0.2)$) -- ($(xn.north) + (0, 0.2)$);
    \draw[dotted] ($(c1.south)+(0, -0.2)$) -- ($(cm.north) + (0, 0.2)$);
    
    \draw[dashed] ($(x1.east)+(0.2,0)$) -- ($(c1.west)+(-0.2,0)$);
    \draw[dashed] ($(x1.east)+(0.2,0)$) -- ($(c1.west)+(-0.2,-1)$);
    \draw[dashed] ($(x1.east)+(0.2,0)$) -- ($(c1.west)+(-0.2,-2)$);
    
    \draw[dashed] ($(xn.east)+(0.2,0)$) -- ($(cm.west)+(-0.2,0)$);
    \draw[dashed] ($(xn.east)+(0.2,0)$) -- ($(cm.west)+(-0.2,1.5)$);
    \draw[dashed] ($(xn.east)+(0.2,0)$) -- ($(cm.west)+(-0.2,2.5)$);
\end{tikzpicture}
\end{subfigure}
\caption{(Left) $\mathbf{H}_{\mathbf{C}}^{2}$. (Right) $\mathbf{H}_{\mathbf{V}}^{2}$}
\label{fig:gcnn_3}
\end{figure} 

As our design will ultimately extract $\mu$ from the transformed variable feature vectors, we first perform the half-convolution over all constraints, and then over all variables. This guarantees that at the representation $\mathbf{H}_{\mathbf{V}}^{2}$, all variables that feature together in a constraint have shared and received information. In a similar manner to the beginning of our forward pass, we now reduce all resulting variable feature vectors using a neural network to dimension 4, i.e. the amount of cut selector parameters. The result $\mathbf{H}_{\mathbf{V}}^{3}$ is shown in Figure \ref{fig:gcnn_4}. Finally, we average over all reduced variable feature vectors, resulting in the mean $\mu \in \reals^{4}$ of our policy. This is a forward pass of the GCNN.

\begin{figure}[h]
\centering
\begin{tikzpicture}[scale=0.75,transform shape]
    \node[draw,circle] (x1) at (9,-2.5) {$x_{1}, [4]$};
    \node[draw,circle] (xn) at (9,-4.5) {$x_{n}, [4]$};
    
    \node[draw,circle] (c1) at (13,-2) {$\text{cons}_{1}, [32]$};
    \node[draw,circle] (cm) at (13,-5) {$\text{cons}_{m}, [32]$};
    
    \draw[dotted] ($(x1.south)+(0, -0.2)$) -- ($(xn.north) + (0, 0.2)$);
    \draw[dotted] ($(c1.south)+(0, -0.2)$) -- ($(cm.north) + (0, 0.2)$);
    
    \draw[dashed] ($(x1.east)+(0.2,0)$) -- ($(c1.west)+(-0.2,0)$);
    \draw[dashed] ($(x1.east)+(0.2,0)$) -- ($(c1.west)+(-0.2,-1)$);
    \draw[dashed] ($(x1.east)+(0.2,0)$) -- ($(c1.west)+(-0.2,-2)$);
    
    \draw[dashed] ($(xn.east)+(0.2,0)$) -- ($(cm.west)+(-0.2,0)$);
    \draw[dashed] ($(xn.east)+(0.2,0)$) -- ($(cm.west)+(-0.2,1.5)$);
    \draw[dashed] ($(xn.east)+(0.2,0)$) -- ($(cm.west)+(-0.2,2.5)$);
\end{tikzpicture}
\caption{$\mathbf{H}_{\mathbf{V}}^{3}$.}
\label{fig:gcnn_4}
\end{figure}

We note that the decision to average the reduced variable feature vectors was inspired by the original design for branching, see \cite{gasse2019exact}. For instance, it would be possible to change the order of the half-convolutions and extract $\mu$ from the transformed constraint feature vectors. This is just one way to change the specific design, with other examples being the type of activation functions, layer structure, and the dimension of each embedding. Our design also makes heavy use of layer normalisation, see \cite{layernorm}, which followed from observations in \cite{gasse2019exact} on improved generalisation capabilities. For our complete design we refer readers to \url{https://github.com/Opt-Mucca/Adaptive-Cutsel-MILP}.

\section{Per Instance Results and Statistics of Section \ref{sec:experiments}}
\label{sec:appendix_experiments}

\begin{table}[h]
\centering
\begin{tabular}{lccccccc}
& \multicolumn{3}{c}{MIPLIB} && \multicolumn{3}{c}{NN-Verification} \\
\cline{2-4} \cline{6-8}
 & Min & Median & Max && Min & Median & Max \\
\hline
Time (s) & 0.001 & 0.037 & 0.944 && 0.032 & 0.087 & 1.378 \\
\hline
\end{tabular}
\caption{Inference time statistics of trained GCNN from Experiment \ref{subsec:adaptive_root} per instance-seed pair.}
\label{tab:inference_time}
\end{table}

\begin{table}[h]
    \centering
    \begin{tabular}{lccccccc}
    & \multicolumn{3}{c}{MIPLIB} && \multicolumn{3}{c}{NN-Verification} \\
    \cline{2-4} \cline{6-8}
    Variable Information & Range & Mean & Median && Range & Mean & Median \\
    \hline
    Num. vars & [154, 23618] & 3251.33 & 1909 && [987, 2101] & 1647.92 & 1653 \\
    Num. bin. vars & [0, 16360], & 1498.03 & 396 && [107, 312] & 183.84 & 183 \\
    Num. int. vars & [0, 5081] & 171.70 & 0 && [0, 0] & 0 & 0 \\
    Num. impl. int. vars & [0, 1] & 0.01 & 0 && [0, 0] & 0 & 0\\
    Num. cont. vars & [0, 23520] & 1581.59 & 393 && [869, 1856] & 1464.08 & 1473 \\
    \hline
    Constraint Information & Range & Mean & Median && Range & Mean & Median \\
    \hline
    Num. cons & [32, 17366] & 2261.62 & 865 && [923, 1933] & 1444.43 & 1452 \\
    Num. linear cons & [0, 13263] & 1181.14 & 306 && [813, 1629] & 1263.82 & 1276 \\
    Num. logicor cons & [0, 17323] & 225.29 & 0 && [0, 0] & 0 & 0 \\
    Num. knapsack cons & [0, 288] & 7.48 & 0 && [0, 0] & 0 & 0 \\
    Num. setppc cons & [0, 11538] & 284.04 & 0 && [0, 0] & 0 & 0 \\
    Num. varbound cons & [0, 5000] & 563.67 & 17 && [106, 304] & 180.61 & 180 \\
    \hline \\
    \end{tabular}
    \caption{Instance statistics for Experiments \ref{subsec:lb_experiment}, \ref{subsec:random_seed}, \ref{subsec:smac}, \ref{subsec:adaptive_root}, and \ref{subsec:tree_experiment}.}
    \label{tab:instance_stats}
\end{table}

\begin{table}
\scriptsize
\begin{minipage}{0.5\textwidth}
\begin{center}
\begin{tabularx}{\textwidth}{l@{\hspace{1.5\tabcolsep}}S@{\hspace{1.5\tabcolsep}}S@{\hspace{1.5\tabcolsep}}S@{\hspace{1.5\tabcolsep}}S@{\hspace{1.5\tabcolsep}}S@{\hspace{1.5\tabcolsep}}c}
Instance & \text{primal-dual} & $\lambda_{1}$ & $\lambda_{2}$ & $\lambda_{3}$ & $\lambda_{4}$ & \#BP \\
\hline
    22433 & 0.22 & 0.0 & 0.1 & 0.7 & 0.2 & 1 \\
    23588 & 0.02 & 0.2 & 0.5 & 0.1 & 0.2 & 2 \\
    50v-10 & 0.04 & 0.1 & 0.0 & 0.4 & 0.5 & 1 \\
    a1c1s1 & 0.25 & 0.6 & 0.1 & 0.3 & 0.0 & 1 \\
    a2c1s1 & 0.36 & 0.6 & 0.1 & 0.3 & 0.0 & 1 \\
    app3 & 0.28 & 0.6 & 0.1 & 0.2 & 0.1 & 1 \\
    b1c1s1 & 0.38 & 0.2 & 0.2 & 0.3 & 0.3 & 1 \\
    b2c1s1 & 0.33 & 0.1 & 0.2 & 0.3 & 0.4 & 1 \\
    beasleyC1 & 0.17 & 0.2 & 0.1 & 0.0 & 0.7 & 1 \\
    beasleyC2 & 0.04 & 0.0 & 0.5 & 0.0 & 0.5 & 1 \\
    berlin & 0.09 & 0.2 & 0.6 & 0.0 & 0.2 & 1 \\
    berlin\_5\_8\_0 & 0.2 & 0.0 & 0.4 & 0.3 & 0.3 & 1 \\
    bg512142 & 0.02 & 0.0 & 0.8 & 0.0 & 0.2 & 1 \\
    bienst1 & 0.29 & 0.2 & 0.0 & 0.6 & 0.2 & 1 \\
    bienst2 & 0.21 & 0.2 & 0.0 & 0.7 & 0.1 & 1 \\
    bppc8-09 & 0.03 & 0.7 & 0.3 & 0.0 & 0.0 & 1 \\
    brasil & 0.05 & 0.5 & 0.0 & 0.2 & 0.3 & 1 \\
    dg012142 & 0.0 & 0.1 & 0.0 & 0.0 & 0.9 & 1 \\
    dws008-01 & 0.03 & 0.2 & 0.7 & 0.1 & 0.0 & 1 \\
    eil33-2 & 0.01 & 0.1 & 0.5 & 0.4 & 0.0 & 1 \\
    exp-1-500-5-5 & 0.69 & 0.3 & 0.0 & 0.3 & 0.4 & 1 \\
    fhnw-schedule-paira100 & 0.01 & 0.0 & 0.2 & 0.7 & 0.1 & 1 \\
    g200x740 & 0.64 & 0.4 & 0.1 & 0.1 & 0.4 & 1 \\
    glass4 & 0.01 & 0.5 & 0.2 & 0.1 & 0.2 & 1 \\
    gmu-35-40 & 0.01 & 0.4 & 0.0 & 0.0 & 0.6 & 1 \\
    graphdraw-domain & 0.01 & 0.2 & 0.0 & 0.8 & 0.0 & 1 \\
    graphdraw-gemcutter & 0.01 & 0.6 & 0.1 & 0.3 & 0.0 & 1 \\
    h50x2450 & 0.54 & 0.0 & 0.0 & 0.0 & 1.0 & 1 \\
    hgms-det & 0.29 & 0.5 & 0.3 & 0.0 & 0.2 & 1 \\
    ic97\_potential & 0.01 & 0.8 & 0.0 & 0.1 & 0.1 & 1 \\
    ic97\_tension & 0.09 & 0.2 & 0.3 & 0.5 & 0.0 & 1 \\
    icir97\_tension & 0.14 & 0.2 & 0.1 & 0.0 & 0.7 & 1 \\
    k16x240b & 0.12 & 0.5 & 0.3 & 0.0 & 0.2 & 1 \\
    lotsize & 0.15 & 0.3 & 0.1 & 0.5 & 0.1 & 1 \\
    mik-250-20-75-2 & 0.11 & 0.3 & 0.4 & 0.0 & 0.3 & 1 \\
    mik-250-20-75-3 & 0.09 & 0.3 & 0.1 & 0.0 & 0.6 & 1 \\
    mik-250-20-75-5 & 0.16 & 0.2 & 0.1 & 0.7 & 0.0 & 1 \\
    milo-v12-6-r2-40-1 & 0.33 & 0.9 & 0.0 & 0.1 & 0.0 & 1 \\
    milo-v13-4-3d-4-0 & 0.01 & 0.7 & 0.2 & 0.0 & 0.1 & 1 \\
    misc07 & 0.06 & 0.8 & 0.0 & 0.1 & 0.1 & 1 \\
    mkc & 0.28 & 0.4 & 0.3 & 0.2 & 0.1 & 1 \\
    n3700 & 0.08 & 0.1 & 0.0 & 0.2 & 0.7 & 1 \\
    n3707 & 0.09 & 0.2 & 0.0 & 0.8 & 0.0 & 1
\end{tabularx}
\end{center}

\end{minipage} \hfill
\begin{minipage}{0.5\textwidth}
\begin{center}
\begin{tabularx}{\textwidth}{l@{\hspace{1.5\tabcolsep}}S@{\hspace{1.5\tabcolsep}}S@{\hspace{1.5\tabcolsep}}S@{\hspace{1.5\tabcolsep}}S@{\hspace{1.5\tabcolsep}}S@{\hspace{1.5\tabcolsep}}c}
Instance & \text{primal-dual} & $\lambda_{1}$ & $\lambda_{2}$ & $\lambda_{3}$ & $\lambda_{4}$ & \#BP \\
\hline
    n370b & 0.08 & 0.3 & 0.1 & 0.6 & 0.0 & 1 \\
    n5-3 & 0.2 & 0.1 & 0.4 & 0.4 & 0.1 & 1 \\
    n7-3 & 0.09 & 0.6 & 0.0 & 0.4 & 0.0 & 1 \\
    n9-3 & 0.16 & 0.1 & 0.2 & 0.3 & 0.4 & 1 \\
    neos-1423785 & 0.04 & 0.2 & 0.0 & 0.0 & 0.8 & 1 \\
    neos-1445738 & 0.0 & 0.2 & 0.4 & 0.2 & 0.2 & 1 \\
    neos-1456979 & 0.21 & 0.5 & 0.0 & 0.0 & 0.5 & 1 \\
    neos-3046601-motu & 0.01 & 0.2 & 0.4 & 0.1 & 0.3 & 1 \\
    neos-3046615-murg & 0.01 & 0.6 & 0.2 & 0.1 & 0.1 & 1 \\
    neos-3381206-awhea & 0.02 & 0.1 & 0.0 & 0.6 & 0.3 & 12 \\
    neos-3426085-ticino & 0.0 & 0.2 & 0.1 & 0.0 & 0.7 & 4 \\
    neos-3530905-gaula & 0.01 & 0.0 & 0.1 & 0.0 & 0.9 & 14 \\
    neos-3627168-kasai & 0.12 & 0.2 & 0.0 & 0.0 & 0.8 & 1 \\
    neos-4333464-siret & 0.0 & 0.0 & 0.4 & 0.3 & 0.3 & 1 \\
    neos-4387871-tavua & 0.02 & 0.2 & 0.6 & 0.0 & 0.2 & 1 \\
    neos-4650160-yukon & 0.02 & 0.1 & 0.1 & 0.0 & 0.8 & 1 \\
    neos-4736745-arroux & 0.07 & 0.3 & 0.1 & 0.1 & 0.5 & 1 \\
    neos-4954672-berkel & 0.25 & 0.9 & 0.0 & 0.0 & 0.1 & 1 \\
    neos-5076235-embley & 0.18 & 0.6 & 0.0 & 0.0 & 0.4 & 4 \\
    neos-5107597-kakapo & 0.01 & 0.0 & 0.4 & 0.0 & 0.6 & 1 \\
    neos-5260764-orauea & 0.02 & 0.5 & 0.1 & 0.4 & 0.0 & 1 \\
    neos-5261882-treska & 0.09 & 0.4 & 0.1 & 0.4 & 0.1 & 1 \\
    neos-631517 & 0.07 & 0.7 & 0.1 & 0.1 & 0.1 & 1 \\
    neos-691058 & 0.64 & 0.3 & 0.0 & 0.2 & 0.5 & 1 \\
    neos-860300 & 0.1 & 0.1 & 0.4 & 0.4 & 0.1 & 1 \\
    neos16 & 0.0 & 0.0 & 0.1 & 0.1 & 0.8 & 3 \\
    newdano & 0.05 & 0.1 & 0.1 & 0.5 & 0.3 & 1 \\
    nexp-150-20-1-5 & 0.28 & 0.1 & 0.0 & 0.1 & 0.8 & 3 \\
    nexp-150-20-8-5 & 0.86 & 0.3 & 0.0 & 0.3 & 0.4 & 1 \\
    p200x1188c & 0.01 & 0.0 & 0.1 & 0.8 & 0.1 & 1 \\
    p500x2988 & 0.11 & 0.6 & 0.2 & 0.1 & 0.1 & 1 \\
    pg & 0.46 & 0.4 & 0.1 & 0.2 & 0.3 & 1 \\
    pg5\_34 & 0.34 & 1.0 & 0.0 & 0.0 & 0.0 & 1 \\
    probportfolio & 0.01 & 0.2 & 0.2 & 0.1 & 0.5 & 1 \\
    prod2 & 0.02 & 0.3 & 0.1 & 0.1 & 0.5 & 1 \\
    r50x360 & 0.09 & 0.2 & 0.2 & 0.6 & 0.0 & 1 \\
    ran13x13 & 0.11 & 0.0 & 0.4 & 0.0 & 0.6 & 1 \\
    supportcase20 & 0.36 & 0.7 & 0.2 & 0.0 & 0.1 & 1 \\
    supportcase26 & 0.03 & 0.7 & 0.1 & 0.2 & 0.0 & 1 \\
    swath & 0.01 & 0.0 & 0.4 & 0.3 & 0.3 & 1 \\
    tanglegram6 & 0.0 & 0.1 & 0.1 & 0.2 & 0.6 & 1 \\
    timtab1CUTS & 0.04 & 0.0 & 0.2 & 0.6 & 0.2 & 1 \\
    tr12-30 & 0.73 & 0.6 & 0.1 & 0.2 & 0.1 & 1 \\
    usAbbrv-8-25\_70 & 0.72 & 0.2 & 0.3 & 0.3 & 0.2 & 1  
\end{tabularx}
\end{center}
\end{minipage}
\bigskip
\caption{Per instance results of Experiment \ref{subsec:lb_experiment} (Grid search). Primal-dual refers to the relative primal-dual difference improvement. $\{\lambda_{1}, \lambda_{2}, \lambda_{3}, \lambda_{4}\}$ are the multipliers for \texttt{dcd}, \texttt{eff}, \texttt{isp}, and \texttt{obp}. \#BP refers to the number of best parameter combinations. In the case of \#BP > 1, a single best choice $\{\lambda_{1}, \lambda_{2}, \lambda_{3}, \lambda_{4}\}$ is provided.}
\label{tab:per_instance_grid}
\end{table}
\begin{table}
\scriptsize
\begin{minipage}{0.5\textwidth}
\begin{center}
\begin{tabularx}{\textwidth}{l@{\hspace{0.5\tabcolsep}}S@{\hspace{0.5\tabcolsep}}S@{\hspace{0.5\tabcolsep}}S@{\hspace{0.5\tabcolsep}}S@{\hspace{0.5\tabcolsep}}S}
Instance & \text{primal-dual} & $\lambda_{1}$ & $\lambda_{2}$ & $\lambda_{3}$ & $\lambda_{4}$ \\
\hline
    22433 & 0.12 & 0.24 & 0.27 & 0.37 & 0.11 \\
    23588 & 0.01 & 0.25 & 0.28 & 0.36 & 0.12 \\
    50v-10 & -0.0 & 0.26 & 0.3 & 0.3 & 0.15 \\
    a1c1s1 & 0.18 & 0.26 & 0.3 & 0.29 & 0.15 \\
    a2c1s1 & 0.25 & 0.26 & 0.3 & 0.29 & 0.15 \\
    app3 & 0.09 & 0.26 & 0.3 & 0.29 & 0.15 \\
    b1c1s1 & 0.28 & 0.26 & 0.3 & 0.29 & 0.15 \\
    b2c1s1 & 0.22 & 0.26 & 0.3 & 0.29 & 0.15 \\
    beasleyC1 & 0.02 & 0.25 & 0.3 & 0.29 & 0.15 \\
    beasleyC2 & -0.03 & 0.25 & 0.3 & 0.29 & 0.15 \\
    berlin & -0.31 & 0.25 & 0.3 & 0.29 & 0.15 \\
    berlin\_5\_8\_0 & 0.0 & 0.25 & 0.28 & 0.33 & 0.13 \\
    bg512142 & 0.01 & 0.26 & 0.3 & 0.3 & 0.15 \\
    bienst1 & -0.05 & 0.26 & 0.3 & 0.29 & 0.15 \\
    bienst2 & 0.01 & 0.26 & 0.3 & 0.29 & 0.15 \\
    bppc8-09 & -0.03 & 0.25 & 0.29 & 0.32 & 0.13 \\
    brasil & -1.0 & 0.25 & 0.3 & 0.29 & 0.15 \\
    dg012142 & -0.0 & 0.26 & 0.3 & 0.29 & 0.15 \\
    dws008-01 & -0.01 & 0.26 & 0.3 & 0.3 & 0.14 \\
    eil33-2 & 0.0 & 0.26 & 0.3 & 0.3 & 0.15 \\
    exp-1-500-5-5 & 0.35 & 0.25 & 0.31 & 0.29 & 0.15 \\
    fhnw-schedule-paira100 & 0.01 & 0.27 & 0.27 & 0.3 & 0.16 \\
    g200x740 & 0.6 & 0.25 & 0.3 & 0.3 & 0.15 \\
    glass4 & 0.0 & 0.25 & 0.28 & 0.32 & 0.14 \\
    gmu-35-40 & -0.0 & 0.25 & 0.29 & 0.32 & 0.14 \\
    graphdraw-domain & 0.0 & 0.24 & 0.29 & 0.36 & 0.12 \\
    graphdraw-gemcutter & -0.0 & 0.24 & 0.29 & 0.35 & 0.12 \\
    h50x2450 & 0.32 & 0.25 & 0.3 & 0.3 & 0.15 \\
    hgms-det & 0.09 & 0.25 & 0.29 & 0.31 & 0.14 \\
    ic97\_potential & -0.0 & 0.26 & 0.29 & 0.3 & 0.15 \\
    ic97\_tension & 0.08 & 0.26 & 0.3 & 0.3 & 0.15 \\
    icir97\_tension & 0.05 & 0.25 & 0.28 & 0.31 & 0.16 \\
    k16x240b & 0.04 & 0.25 & 0.3 & 0.3 & 0.15 \\
    lotsize & 0.06 & 0.26 & 0.3 & 0.3 & 0.15 \\
    mik-250-20-75-2 & -0.08 & 0.25 & 0.29 & 0.3 & 0.16 \\
    mik-250-20-75-3 & -0.01 & 0.25 & 0.29 & 0.3 & 0.16 \\
    mik-250-20-75-5 & 0.05 & 0.25 & 0.29 & 0.3 & 0.16 \\
    milo-v12-6-r2-40-1 & 0.23 & 0.26 & 0.3 & 0.3 & 0.15 \\
    milo-v13-4-3d-4-0 & 0.01 & 0.25 & 0.31 & 0.3 & 0.14 \\
    misc07 & -0.0 & 0.26 & 0.28 & 0.33 & 0.13 \\
    mkc & 0.04 & 0.26 & 0.29 & 0.3 & 0.15 \\
    n3700 & 0.07 & 0.26 & 0.29 & 0.3 & 0.15 \\
    n3707 & 0.08 & 0.26 & 0.29 & 0.3 & 0.15
\end{tabularx}
\end{center}

\end{minipage} \hfill
\begin{minipage}{0.5\textwidth}
\begin{center}
\begin{tabularx}{\textwidth}{l@{\hspace{0.5\tabcolsep}}S@{\hspace{0.5\tabcolsep}}S@{\hspace{0.5\tabcolsep}}S@{\hspace{0.5\tabcolsep}}S@{\hspace{0.5\tabcolsep}}S}
Instance & \text{primal-dual} & $\lambda_{1}$ & $\lambda_{2}$ & $\lambda_{3}$ & $\lambda_{4}$\\
\hline
    n370b & 0.08 & 0.26 & 0.29 & 0.3 & 0.15 \\
    n5-3 & 0.0 & 0.25 & 0.32 & 0.28 & 0.15 \\
    n7-3 & 0.05 & 0.25 & 0.32 & 0.28 & 0.15 \\
    n9-3 & 0.13 & 0.25 & 0.32 & 0.28 & 0.15 \\
    neos-1423785 & -0.0 & 0.26 & 0.3 & 0.29 & 0.15 \\
    neos-1445738 & -0.0 & 0.25 & 0.32 & 0.28 & 0.15 \\
    neos-1456979 & 0.04 & 0.25 & 0.29 & 0.32 & 0.14 \\
    neos-3046601-motu & -0.0 & 0.26 & 0.28 & 0.31 & 0.15 \\
    neos-3046615-murg & -0.0 & 0.25 & 0.28 & 0.31 & 0.15 \\
    neos-3381206-awhea & 0.0 & 0.24 & 0.28 & 0.31 & 0.17 \\
    neos-3426085-ticino & -0.01 & 0.24 & 0.28 & 0.32 & 0.16 \\
    neos-3530905-gaula & -0.01 & 0.24 & 0.28 & 0.31 & 0.17 \\
    neos-3627168-kasai & 0.05 & 0.25 & 0.31 & 0.3 & 0.15 \\
    neos-4333464-siret & -0.03 & 0.24 & 0.3 & 0.32 & 0.13 \\
    neos-4387871-tavua & -0.02 & 0.25 & 0.3 & 0.32 & 0.14 \\
    neos-4650160-yukon & 0.01 & 0.25 & 0.3 & 0.3 & 0.15 \\
    neos-4736745-arroux & -0.06 & 0.24 & 0.28 & 0.32 & 0.16 \\
    neos-4954672-berkel & 0.14 & 0.26 & 0.3 & 0.29 & 0.15 \\
    neos-5076235-embley & -0.07 & 0.25 & 0.31 & 0.29 & 0.15 \\
    neos-5107597-kakapo & -0.0 & 0.27 & 0.27 & 0.31 & 0.16 \\
    neos-5260764-orauea & 0.01 & 0.26 & 0.29 & 0.31 & 0.14 \\
    neos-5261882-treska & 0.06 & 0.25 & 0.29 & 0.32 & 0.14 \\
    neos-631517 & 0.05 & 0.25 & 0.28 & 0.33 & 0.14 \\
    neos-691058 & 0.34 & 0.25 & 0.29 & 0.31 & 0.14 \\
    neos-860300 & -0.05 & 0.26 & 0.28 & 0.33 & 0.13 \\
    neos16 & 0.0 & 0.25 & 0.27 & 0.33 & 0.15 \\
    newdano & 0.02 & 0.26 & 0.3 & 0.3 & 0.15 \\
    nexp-150-20-1-5 & 0.22 & 0.25 & 0.3 & 0.3 & 0.15 \\
    nexp-150-20-8-5 & 0.67 & 0.25 & 0.29 & 0.32 & 0.14 \\
    p200x1188c & -0.0 & 0.25 & 0.3 & 0.29 & 0.15 \\
    p500x2988 & -0.02 & 0.25 & 0.3 & 0.29 & 0.15 \\
    pg & 0.16 & 0.26 & 0.3 & 0.3 & 0.14 \\
    pg5\_34 & 0.12 & 0.25 & 0.32 & 0.28 & 0.15 \\
    probportfolio & 0.0 & 0.27 & 0.27 & 0.29 & 0.16 \\
    prod2 & 0.0 & 0.26 & 0.3 & 0.3 & 0.14 \\
    r50x360 & 0.05 & 0.25 & 0.3 & 0.3 & 0.15 \\
    ran13x13 & -0.01 & 0.26 & 0.29 & 0.3 & 0.15 \\
    supportcase20 & 0.19 & 0.26 & 0.3 & 0.29 & 0.15 \\
    supportcase26 & 0.01 & 0.26 & 0.27 & 0.3 & 0.16 \\
    swath & -0.0 & 0.26 & 0.29 & 0.31 & 0.14 \\
    tanglegram6 & -0.0 & 0.25 & 0.28 & 0.35 & 0.11 \\
    timtab1CUTS & -0.0 & 0.25 & 0.3 & 0.3 & 0.15 \\
    tr12-30 & 0.62 & 0.26 & 0.3 & 0.3 & 0.15 \\
    usAbbrv-8-25\_70 & 0.47 & 0.25 & 0.28 & 0.33 & 0.14 
\end{tabularx}
\end{center}
\end{minipage}
\bigskip
\caption{Per instance results of Experiment \ref{subsec:random_seed} (Random seed). Primal-dual refers to the relative primal-dual difference improvement. $\{\lambda_{1}, \lambda_{2}, \lambda_{3}, \lambda_{4}\}$ are the multipliers for \texttt{dcd}, \texttt{eff}, \texttt{isp}, and \texttt{obp}.}
\label{tab:per_instance_random}
\end{table}
\begin{table}
\scriptsize
\begin{minipage}{0.5\textwidth}
\begin{center}
\begin{tabularx}{\textwidth}{l@{\hspace{0.5\tabcolsep}}S@{\hspace{0.5\tabcolsep}}S@{\hspace{0.5\tabcolsep}}S@{\hspace{0.5\tabcolsep}}S@{\hspace{0.5\tabcolsep}}S}
Instance & \text{primal-dual} & $\lambda_{1}$ & $\lambda_{2}$ & $\lambda_{3}$ & $\lambda_{4}$\\
\hline
    22433 & 0.03 & 0.6 & 0.12 & 0.18 & 0.1 \\
    23588 & -0.0 & 0.6 & 0.12 & 0.18 & 0.1 \\
    50v-10 & -0.01 & 0.6 & 0.12 & 0.18 & 0.1 \\
    a1c1s1 & 0.2 & 0.6 & 0.12 & 0.18 & 0.1 \\
    a2c1s1 & 0.28 & 0.6 & 0.12 & 0.18 & 0.1 \\
    app3 & 0.15 & 0.6 & 0.12 & 0.18 & 0.1 \\
    b1c1s1 & 0.28 & 0.6 & 0.12 & 0.18 & 0.1 \\
    b2c1s1 & 0.27 & 0.6 & 0.12 & 0.18 & 0.1 \\
    beasleyC1 & 0.06 & 0.6 & 0.12 & 0.18 & 0.1 \\
    beasleyC2 & -0.04 & 0.6 & 0.12 & 0.18 & 0.1 \\
    berlin & 0.03 & 0.6 & 0.12 & 0.18 & 0.1 \\
    berlin\_5\_8\_0 & 0.0 & 0.6 & 0.12 & 0.18 & 0.1 \\
    bg512142 & 0.0 & 0.6 & 0.12 & 0.18 & 0.1 \\
    bienst1 & 0.11 & 0.6 & 0.12 & 0.18 & 0.1 \\
    bienst2 & 0.03 & 0.6 & 0.12 & 0.18 & 0.1 \\
    bppc8-09 & -0.01 & 0.6 & 0.12 & 0.18 & 0.1 \\
    brasil & 0.02 & 0.6 & 0.12 & 0.18 & 0.1 \\
    dg012142 & 0.0 & 0.6 & 0.12 & 0.18 & 0.1 \\
    dws008-01 & -0.02 & 0.6 & 0.12 & 0.18 & 0.1 \\
    eil33-2 & 0.0 & 0.6 & 0.12 & 0.18 & 0.1 \\
    exp-1-500-5-5 & 0.58 & 0.6 & 0.12 & 0.18 & 0.1 \\
    fhnw-schedule-paira100 & 0.0 & 0.6 & 0.12 & 0.18 & 0.1 \\
    g200x740 & 0.58 & 0.6 & 0.12 & 0.18 & 0.1 \\
    glass4 & -0.01 & 0.6 & 0.12 & 0.18 & 0.1 \\
    gmu-35-40 & 0.0 & 0.6 & 0.12 & 0.18 & 0.1 \\
    graphdraw-domain & 0.0 & 0.6 & 0.12 & 0.18 & 0.1 \\
    graphdraw-gemcutter & 0.0 & 0.6 & 0.12 & 0.18 & 0.1 \\
    h50x2450 & 0.32 & 0.6 & 0.12 & 0.18 & 0.1 \\
    hgms-det & 0.06 & 0.6 & 0.12 & 0.18 & 0.1 \\
    ic97\_potential & -0.0 & 0.6 & 0.12 & 0.18 & 0.1 \\
    ic97\_tension & 0.08 & 0.6 & 0.12 & 0.18 & 0.1 \\
    icir97\_tension & 0.04 & 0.6 & 0.12 & 0.18 & 0.1 \\
    k16x240b & 0.08 & 0.6 & 0.12 & 0.18 & 0.1 \\
    lotsize & 0.09 & 0.6 & 0.12 & 0.18 & 0.1 \\
    mik-250-20-75-2 & 0.05 & 0.6 & 0.12 & 0.18 & 0.1 \\
    mik-250-20-75-3 & -0.03 & 0.6 & 0.12 & 0.18 & 0.1 \\
    mik-250-20-75-5 & 0.13 & 0.6 & 0.12 & 0.18 & 0.1 \\
    milo-v12-6-r2-40-1 & 0.28 & 0.6 & 0.12 & 0.18 & 0.1 \\
    milo-v13-4-3d-4-0 & 0.01 & 0.6 & 0.12 & 0.18 & 0.1 \\
    misc07 & 0.0 & 0.6 & 0.12 & 0.18 & 0.1 \\
    mkc & 0.24 & 0.6 & 0.12 & 0.18 & 0.1 \\
    n3700 & 0.07 & 0.6 & 0.12 & 0.18 & 0.1 \\
    n3707 & 0.07 & 0.6 & 0.12 & 0.18 & 0.1 
\end{tabularx}
\end{center}

\end{minipage} \hfill
\begin{minipage}{0.5\textwidth}
\begin{center}
\begin{tabularx}{\textwidth}{l@{\hspace{0.5\tabcolsep}}S@{\hspace{0.5\tabcolsep}}S@{\hspace{0.5\tabcolsep}}S@{\hspace{0.5\tabcolsep}}S@{\hspace{0.5\tabcolsep}}S}
Instance & \text{primal-dual} & $\lambda_{1}$ & $\lambda_{2}$ & $\lambda_{3}$ & $\lambda_{4}$\\
    \hline
    n370b & 0.07 & 0.6 & 0.12 & 0.18 & 0.1 \\
    n5-3 & 0.14 & 0.6 & 0.12 & 0.18 & 0.1 \\
    n7-3 & 0.05 & 0.6 & 0.12 & 0.18 & 0.1 \\
    n9-3 & 0.11 & 0.6 & 0.12 & 0.18 & 0.1 \\
    neos-1423785 & -0.01 & 0.6 & 0.12 & 0.18 & 0.1 \\
    neos-1445738 & -0.0 & 0.6 & 0.12 & 0.18 & 0.1 \\
    neos-1456979 & 0.1 & 0.6 & 0.12 & 0.18 & 0.1 \\
    neos-3046601-motu & -0.0 & 0.6 & 0.12 & 0.18 & 0.1 \\
    neos-3046615-murg & -0.0 & 0.6 & 0.12 & 0.18 & 0.1 \\
    neos-3381206-awhea & 0.0 & 0.6 & 0.12 & 0.18 & 0.1 \\
    neos-3426085-ticino & -0.01 & 0.6 & 0.12 & 0.18 & 0.1 \\
    neos-3530905-gaula & -0.01 & 0.6 & 0.12 & 0.18 & 0.1 \\
    neos-3627168-kasai & 0.06 & 0.6 & 0.12 & 0.18 & 0.1 \\
    neos-4333464-siret & -0.12 & 0.6 & 0.12 & 0.18 & 0.1 \\
    neos-4387871-tavua & 0.0 & 0.6 & 0.12 & 0.18 & 0.1 \\
    neos-4650160-yukon & 0.0 & 0.6 & 0.12 & 0.18 & 0.1 \\
    neos-4736745-arroux & -0.04 & 0.6 & 0.12 & 0.18 & 0.1 \\
    neos-4954672-berkel & 0.19 & 0.6 & 0.12 & 0.18 & 0.1 \\
    neos-5076235-embley & 0.13 & 0.6 & 0.12 & 0.18 & 0.1 \\
    neos-5107597-kakapo & -0.0 & 0.6 & 0.12 & 0.18 & 0.1 \\
    neos-5260764-orauea & 0.01 & 0.6 & 0.12 & 0.18 & 0.1 \\
    neos-5261882-treska & 0.05 & 0.6 & 0.12 & 0.18 & 0.1 \\
    neos-631517 & 0.06 & 0.6 & 0.12 & 0.18 & 0.1 \\
    neos-691058 & 0.41 & 0.6 & 0.12 & 0.18 & 0.1 \\
    neos-860300 & -0.02 & 0.6 & 0.12 & 0.18 & 0.1 \\
    neos16 & 0.0 & 0.6 & 0.12 & 0.18 & 0.1 \\
    newdano & -0.03 & 0.6 & 0.12 & 0.18 & 0.1 \\
    nexp-150-20-1-5 & 0.27 & 0.6 & 0.12 & 0.18 & 0.1 \\
    nexp-150-20-8-5 & 0.75 & 0.6 & 0.12 & 0.18 & 0.1 \\
    p200x1188c & 0.0 & 0.6 & 0.12 & 0.18 & 0.1 \\
    p500x2988 & -0.0 & 0.6 & 0.12 & 0.18 & 0.1 \\
    pg & 0.17 & 0.6 & 0.12 & 0.18 & 0.1 \\
    pg5\_34 & 0.16 & 0.6 & 0.12 & 0.18 & 0.1 \\
    probportfolio & 0.01 & 0.6 & 0.12 & 0.18 & 0.1 \\
    prod2 & -0.02 & 0.6 & 0.12 & 0.18 & 0.1 \\
    r50x360 & 0.06 & 0.6 & 0.12 & 0.18 & 0.1 \\
    ran13x13 & 0.01 & 0.6 & 0.12 & 0.18 & 0.1 \\
    supportcase20 & 0.26 & 0.6 & 0.12 & 0.18 & 0.1 \\
    supportcase26 & 0.02 & 0.6 & 0.12 & 0.18 & 0.1 \\
    swath & -0.0 & 0.6 & 0.12 & 0.18 & 0.1 \\
    tanglegram6 & 0.0 & 0.6 & 0.12 & 0.18 & 0.1 \\
    timtab1CUTS & -0.0 & 0.6 & 0.12 & 0.18 & 0.1 \\
    tr12-30 & 0.73 & 0.6 & 0.12 & 0.18 & 0.1 \\
    usAbbrv-8-25\_70 & 0.72 & 0.6 & 0.12 & 0.18 & 0.1  
\end{tabularx}
\end{center}
\end{minipage}
\bigskip
\caption{Per instance results of Experiment \ref{subsec:smac} (Standard Learning Method - SMAC). Primal-dual refers to the relative primal-dual difference improvement. $\{\lambda_{1}, \lambda_{2}, \lambda_{3}, \lambda_{4}\}$ are the multipliers for \texttt{dcd}, \texttt{eff}, \texttt{isp}, and \texttt{obp}.}
\label{tab:per_instance_smac}
\end{table}
\begin{table}
\scriptsize
\begin{minipage}{0.5\textwidth}
\begin{center}
\begin{tabularx}{\textwidth}{l@{\hspace{0.5\tabcolsep}}S@{\hspace{0.5\tabcolsep}}S@{\hspace{0.5\tabcolsep}}S@{\hspace{0.5\tabcolsep}}S@{\hspace{0.5\tabcolsep}}S}
Instance & \text{primal-dual} & $\lambda_{1}$ & $\lambda_{2}$ & $\lambda_{3}$ & $\lambda_{4}$\\
\hline
    22433 & 0.12 & 0.3 & 0.21 & 0.27 & 0.21 \\
    23588 & 0.01 & 0.31 & 0.21 & 0.27 & 0.22 \\
    50v-10 & 0.01 & 0.39 & 0.16 & 0.23 & 0.22 \\
    a1c1s1 & 0.22 & 0.37 & 0.12 & 0.27 & 0.24 \\
    a2c1s1 & 0.34 & 0.37 & 0.12 & 0.27 & 0.24 \\
    app3 & 0.18 & 0.37 & 0.13 & 0.26 & 0.24 \\
    b1c1s1 & 0.31 & 0.37 & 0.12 & 0.27 & 0.24 \\
    b2c1s1 & 0.29 & 0.37 & 0.12 & 0.27 & 0.24 \\
    beasleyC1 & 0.03 & 0.6 & 0.08 & 0.24 & 0.07 \\
    beasleyC2 & -0.08 & 0.59 & 0.09 & 0.25 & 0.07 \\
    berlin & 0.03 & 0.58 & 0.1 & 0.24 & 0.08 \\
    berlin\_5\_8\_0 & 0.0 & 0.32 & 0.16 & 0.3 & 0.22 \\
    bg512142 & -0.0 & 0.36 & 0.13 & 0.27 & 0.23 \\
    bienst1 & -0.03 & 0.36 & 0.13 & 0.26 & 0.25 \\
    bienst2 & 0.04 & 0.37 & 0.13 & 0.26 & 0.24 \\
    bppc8-09 & -0.01 & 0.39 & 0.19 & 0.21 & 0.21 \\
    brasil & 0.02 & 0.58 & 0.1 & 0.24 & 0.08 \\
    dg012142 & 0.0 & 0.36 & 0.13 & 0.27 & 0.23 \\
    dws008-01 & -0.01 & 0.38 & 0.17 & 0.23 & 0.22 \\
    eil33-2 & 0.0 & 0.61 & 0.39 & 0.0 & 0.0 \\
    exp-1-500-5-5 & 0.62 & 0.39 & 0.12 & 0.27 & 0.21 \\
    fhnw-schedule-paira100 & 0.01 & 0.35 & 0.16 & 0.26 & 0.24 \\
    g200x740 & 0.56 & 0.55 & 0.11 & 0.24 & 0.1 \\
    glass4 & -0.01 & 0.37 & 0.2 & 0.23 & 0.2 \\
    gmu-35-40 & 0.01 & 0.39 & 0.22 & 0.19 & 0.19 \\
    graphdraw-domain & 0.0 & 0.32 & 0.19 & 0.28 & 0.21 \\
    graphdraw-gemcutter & 0.01 & 0.32 & 0.19 & 0.28 & 0.21 \\
    h50x2450 & 0.4 & 0.43 & 0.15 & 0.23 & 0.19 \\
    hgms-det & 0.18 & 0.36 & 0.18 & 0.24 & 0.23 \\
    ic97\_potential & -0.0 & 0.37 & 0.14 & 0.26 & 0.23 \\
    ic97\_tension & 0.08 & 0.34 & 0.16 & 0.26 & 0.24 \\
    icir97\_tension & 0.08 & 0.3 & 0.19 & 0.26 & 0.25 \\
    k16x240b & 0.03 & 0.52 & 0.11 & 0.25 & 0.11 \\
    lotsize & 0.09 & 0.36 & 0.15 & 0.25 & 0.23 \\
    mik-250-20-75-2 & -0.02 & 0.33 & 0.22 & 0.2 & 0.25 \\
    mik-250-20-75-3 & 0.02 & 0.33 & 0.22 & 0.2 & 0.25 \\
    mik-250-20-75-5 & 0.01 & 0.33 & 0.22 & 0.2 & 0.25 \\
    milo-v12-6-r2-40-1 & 0.23 & 0.37 & 0.14 & 0.26 & 0.23 \\
    milo-v13-4-3d-4-0 & -0.0 & 0.39 & 0.12 & 0.27 & 0.22 \\
    misc07 & 0.02 & 0.57 & 0.32 & 0.1 & 0.01 \\
    mkc & 0.18 & 0.39 & 0.21 & 0.19 & 0.2 \\
    n3700 & 0.07 & 0.41 & 0.16 & 0.23 & 0.21 \\
    n3707 & 0.08 & 0.41 & 0.16 & 0.23 & 0.2
\end{tabularx}
\end{center}

\end{minipage} \hfill
\begin{minipage}{0.5\textwidth}
\begin{center}
\begin{tabularx}{\textwidth}{l@{\hspace{0.5\tabcolsep}}S@{\hspace{0.5\tabcolsep}}S@{\hspace{0.5\tabcolsep}}S@{\hspace{0.5\tabcolsep}}S@{\hspace{0.5\tabcolsep}}S}
Instance & \text{primal-dual} & $\lambda_{1}$ & $\lambda_{2}$ & $\lambda_{3}$ & $\lambda_{4}$\\
    \hline
    n370b & 0.08 & 0.41 & 0.16 & 0.23 & 0.2 \\
    n5-3 & 0.07 & 0.43 & 0.13 & 0.24 & 0.19 \\
    n7-3 & 0.08 & 0.44 & 0.13 & 0.25 & 0.19 \\
    n9-3 & 0.07 & 0.44 & 0.13 & 0.24 & 0.19 \\
    neos-1423785 & -0.01 & 0.37 & 0.12 & 0.27 & 0.24 \\
    neos-1445738 & -0.0 & 0.46 & 0.13 & 0.23 & 0.18 \\
    neos-1456979 & 0.1 & 0.36 & 0.2 & 0.23 & 0.22 \\
    neos-3046601-motu & -0.0 & 0.44 & 0.24 & 0.16 & 0.16 \\
    neos-3046615-murg & -0.0 & 0.44 & 0.24 & 0.16 & 0.16 \\
    neos-3381206-awhea & 0.0 & 0.28 & 0.23 & 0.23 & 0.25 \\
    neos-3426085-ticino & 0.0 & 0.28 & 0.23 & 0.25 & 0.25 \\
    neos-3530905-gaula & 0.01 & 0.28 & 0.24 & 0.24 & 0.25 \\
    neos-3627168-kasai & 0.06 & 0.43 & 0.12 & 0.26 & 0.18 \\
    neos-4333464-siret & -0.08 & 0.47 & 0.18 & 0.24 & 0.11 \\
    neos-4387871-tavua & -0.0 & 0.43 & 0.17 & 0.24 & 0.15 \\
    neos-4650160-yukon & 0.0 & 0.39 & 0.15 & 0.25 & 0.21 \\
    neos-4736745-arroux & -0.02 & 0.27 & 0.22 & 0.26 & 0.25 \\
    neos-4954672-berkel & 0.18 & 0.36 & 0.14 & 0.25 & 0.25 \\
    neos-5076235-embley & 0.01 & 0.39 & 0.13 & 0.25 & 0.23 \\
    neos-5107597-kakapo & -0.0 & 0.36 & 0.15 & 0.26 & 0.23 \\
    neos-5260764-orauea & 0.01 & 0.61 & 0.28 & 0.08 & 0.03 \\
    neos-5261882-treska & 0.06 & 0.38 & 0.2 & 0.21 & 0.2 \\
    neos-631517 & 0.05 & 0.35 & 0.2 & 0.24 & 0.21 \\
    neos-691058 & 0.4 & 0.37 & 0.16 & 0.25 & 0.22 \\
    neos-860300 & -0.04 & 0.59 & 0.31 & 0.1 & 0.0 \\
    neos16 & 0.0 & 0.3 & 0.21 & 0.27 & 0.22 \\
    newdano & -0.0 & 0.37 & 0.13 & 0.26 & 0.24 \\
    nexp-150-20-1-5 & 0.27 & 0.59 & 0.09 & 0.26 & 0.05 \\
    nexp-150-20-8-5 & 0.56 & 0.4 & 0.2 & 0.21 & 0.19 \\
    p200x1188c & 0.0 & 0.57 & 0.11 & 0.24 & 0.08 \\
    p500x2988 & 0.03 & 0.56 & 0.11 & 0.24 & 0.09 \\
    pg & 0.23 & 0.35 & 0.14 & 0.25 & 0.25 \\
    pg5\_34 & 0.15 & 0.42 & 0.12 & 0.26 & 0.2 \\
    probportfolio & 0.01 & 0.43 & 0.24 & 0.15 & 0.18 \\
    prod2 & 0.0 & 0.64 & 0.28 & 0.08 & 0.0 \\
    r50x360 & 0.03 & 0.54 & 0.11 & 0.24 & 0.11 \\
    ran13x13 & 0.04 & 0.45 & 0.14 & 0.23 & 0.18 \\
    supportcase20 & 0.36 & 0.53 & 0.11 & 0.24 & 0.12 \\
    supportcase26 & 0.01 & 0.43 & 0.16 & 0.22 & 0.19 \\
    swath & -0.0 & 0.63 & 0.29 & 0.08 & 0.0 \\
    tanglegram6 & -0.0 & 0.51 & 0.3 & 0.19 & 0.0 \\
    timtab1CUTS & 0.0 & 0.36 & 0.15 & 0.27 & 0.23 \\
    tr12-30 & 0.69 & 0.36 & 0.15 & 0.26 & 0.24 \\
    usAbbrv-8-25\_70 & 0.72 & 0.31 & 0.17 & 0.3 & 0.22 
\end{tabularx}
\end{center}
\end{minipage}
\bigskip
\caption{Per instance results of Experiment \ref{subsec:adaptive_root} (Learning Adaptive Parameters). Primal-dual refers to the relative primal-dual difference improvement. $\{\lambda_{1}, \lambda_{2}, \lambda_{3}, \lambda_{4}\}$ are the multipliers for \texttt{dcd}, \texttt{eff}, \texttt{isp}, and \texttt{obp}.}
\label{tab:per_instance_adaptive_root}
\end{table}

\end{document}